\documentclass[11pt]{amsart}

\usepackage{epigamath-voisin}

\usepackage[english]{babel}

\numberwithin{equation}{section}

\usepackage[shortlabels]{enumitem}
\setlist[enumerate,1]{label={\rm(\roman*)}, ref={\rm\roman*}} 

\usepackage{amsmath}
\usepackage{amssymb}
\usepackage{amscd}
\usepackage{amsthm}
\usepackage{graphicx}
\usepackage{fullpage}
\usepackage{color}
\usepackage[cmtip,all,matrix,arrow,tips,curve]{xy}

\newtheorem{theorem}{Theorem}[section]
\newtheorem{proposition}[theorem]{Proposition}
\newtheorem{lemma}[theorem]{Lemma}

\newtheorem{corollary}[theorem]{Corollary}

\theoremstyle{definition}
\newtheorem{definition}[theorem]{Definition}

\theoremstyle{remark}
\newtheorem{remark}[theorem]{Remark}
\newtheorem{example}[theorem]{Example}

\newcommand{\Dis}{\displaystyle}

\def\Zee{\mathbb{Z}}
\def\Q{\mathbb{Q}}

\def\Cee{\mathbb{C}}
\def\Pee{\mathbb{P}}

\def\Ker{\operatorname{Ker}}
\def\Coker{\operatorname{Coker}}

\def\Ext{\operatorname{Ext}}

\def\Pic{\operatorname{Pic}}
\def\Gr{\operatorname{Gr}}
\def\Sym{\operatorname{Sym}}

\def\Spec{\operatorname{Spec}}

\def\scrO{\mathcal{O}}
\def\spcheck{^{\vee}}
\def\hX{\widehat{X}}
\def\hY{\widehat{Y}}

\newcommand{\lra}{\longrightarrow}

\YearArticle{2024} \EpigaArticleNr{18} \ReceivedOn{February 2, 2023}
\InFinalFormOn{May 7, 2024}
\AcceptedOn{May 31, 2024}

\title{Deformations of some local Calabi--Yau manifolds}
\titlemark{Deformations of some local Calabi--Yau manifolds}

\author{Robert Friedman}
\address{Columbia University, Department of Mathematics, New York, NY 10027, USA}
\email{rf@math.columbia.edu}
\author{Radu Laza}
\address{Stony Brook University, Department of Mathematics, Stony Brook, NY 11794, USA}
\email{radu.laza@stonybrook.edu}

\authormark{R.~Friedman and R.~Laza}
 
\AbstractInEnglish{We study  deformations of certain crepant resolutions of isolated rational Gorenstein singularities. After a general discussion of the deformation theory, we specialize to dimension $3$ and consider examples which are good (log) resolutions as well as the case of small resolutions.  We obtain some partial results on the classification of canonical threefold singularities that admit good crepant resolutions. Finally, we study a noncrepant example, the blowup of a small resolution whose exceptional set is a smooth curve.}



\acknowledgement{Research of the second author is supported in part by NSF grant DMS-2101640}

\begin{document}



\maketitle

\begin{prelims}

\DisplayAbstractInEnglish

\bigskip

\DisplayKeyWords

\medskip

\DisplayMSCclass

\end{prelims}


\newpage

\setcounter{tocdepth}{1}

\tableofcontents


\section{Introduction}

This paper is part of a series \cite{FL,FL23,FL23b} investigating the deformation theory of singular Calabi--Yau varieties, \textit{i.e.}  compact analytic spaces $Y$  with isolated Gorenstein canonical singularities such that $\omega_Y\cong \scrO_Y$, building on and generalizing previously known results due to  Kawamata, Namikawa, Namikawa--Steenbrink, and others; see \cite{F, Kawamata, NS,  namtop,namstrata}. The deformation theory  of $Y$ as studied in \cite{FL} has both a local and a global aspect. Locally, if $x\in Y$ is a singular point, we can study the deformation functor $\mathbf{Def}_{(Y,x)}$ of the germ $(Y,x)$. In particular, the tangent space to this functor, or equivalently the ``first-order deformations,'' \textit{i.e.}~the deformations over the dual numbers $\Spec \Cee[\varepsilon]$, are classified by a finite-dimensional vector space $T^1_{Y,x}$, or equivalently by the corresponding skyscraper sheaf supported at $x$. Globally, for a compact analytic space $Y$, there is  the deformation functor  $\mathbf{Def}_Y$, whose tangent space is a finite-dimensional vector space that we   denote by $\mathbb{T}^1_Y$. There is a corresponding sheaf   $T^1_Y$, which is supported on the singular locus of $Y$. In case the singularities of $Y$ are isolated,  $\mathbb{T}^1_Y$ is a skyscraper sheaf supported at the singular points and the stalk of $\mathbb{T}^1_Y$ at $x$ is the vector space $T^1_{Y,x}$. There is a natural morphism of deformation functors $\mathbf{Def}_Y \to \prod_{x\in Y_{\text{\rm{sing}}}} \mathbf{Def}_{(Y,x)}$. Note that $\mathbf{Def}_Y$ and $\mathbf{Def}_{(Y,x)}$ are pro-represented by germs of analytic spaces and the morphism of functors corresponds to a morphism of germs of analytic spaces. On Zariski tangent spaces, the differential of this morphism of functors or germs of spaces becomes a homomorphism $\mathbb{T}^1_Y \to H^0(Y; T^1_Y) = \bigoplus_{x\in Y_{\text{\rm{sing}}}}   T^1_{Y,x}$.   For $\dim Y \ge 3$, this morphism is almost never surjective.  It is thus important  to identify interesting tangent directions in $\bigoplus_{x\in Y_{\text{\rm{sing}}}}  T^1_{Y,x}$ and try to lift these to $\mathbb{T}^1_Y$. If the deformations of $Y$ are unobstructed, such first-order deformations of $Y$ will come from actual deformations. 

Taking a local point of view, let $(X,x)$ be the germ of an isolated Gorenstein canonical singularity (or equivalently an isolated rational Gorenstein  singularity; see  \cite{KollarMori}, \cite[Section~11.1]{Kollar}). We will usually take $X$   to be a \textsl{good Stein representative} for the germ $(X,x)$, i.e.\ a contractible Stein representative with a unique singular point $x$. Let $\pi\colon \hX\to X$ be a  \textsl{good resolution} or a  \textsl{log resolution}, i.e.\  $\pi^{-1}(x) = E$ is a divisor with simple normal crossings. The assumption of Gorenstein canonical singularities means that $K_{\hX} = \pi^*\omega_X\otimes \scrO_{\hX}(D)$ for some effective divisor $D$ on $\hX$. If $D =0$, we say that $\pi$ is a \textsl{good crepant resolution} of $X$. Typical examples of singularities that admit good crepant resolutions are the $O_{16}$ singularities,  i.e.\  affine cones over a smooth cubic surface in $\Pee^3$. More general examples are singularities which are analytically isomorphic to cones over   Fano manifolds embedded via the anticanonical line bundle, for example the cone over a smooth hypersurface of degree $n$ in $\Pee^n$. A related   case is that where there exists a \textsl{small resolution} $\pi'\colon X'\to X$, i.e.\ a resolution where the fiber of $\pi'$ over $x$ has dimension $1$ (more generally, one could also consider the case where the exceptional set  has dimension less that $\dim X -1$). However, at least when $X$ is a local complete intersection and $\dim X \ge 3$, such resolutions can only exist for $\dim X =3$ (see Remark~\ref{smallres}).   Note that  small resolutions $\pi'\colon X' \to X$ are automatically crepant, in the sense that $K_{X'}\cong \scrO_{X'}$. Examples of  such singularities are  $A_{2k-1}$ singularities  in dimension $3$: these are   locally defined by the equation $x^2+ y^2+ z^2 + w^{2k}$. A singularity $X$ admitting a small resolution is terminal but not (locally) $\Q$-factorial. However, most   canonical singularities do not admit good crepant or small resolutions. For instance,  terminal $\Q$-factorial singularities, such as the $A_{2k}$ singularities  in dimension $3$ (\textit{i.e.} those given by $x^2+ y^2+ z^2 + w^{2k+1}$, $k\geq 1$), do not have either a good crepant or a small resolution.

There is   a natural subspace of $H^0(X; T^1_X) $ defined as follows: Let $U = X -\{x\} = \hX -E$, where as above $\pi\colon \hX\to X$ is a good resolution and $E = \pi^{-1}(x)$. In case $\dim X \ge 3$, a theorem of Schlessinger  implies that $H^0(X; T^1_X) \cong H^1(U; T_U)$, where $T_U$ is the tangent bundle of $U$; see \cite[Theorem 2]{Schlessinger}. By Wahl's theory, see \cite{Wahl},  there is a morphism of functors $\mathbf{Def}_{\hX} \to \mathbf{Def}_X$ where the induced map on tangent spaces is the natural restriction map
$$H^1\left(\hX;T_{\hX}\right) \lra H^1\left(U; T_U\right) \cong H^0\left(X; T^1_X\right).$$ 
Informally, we can think of the image of $\mathbf{Def}_{\hX}$ as the \textsl{simultaneous resolution locus}, i.e.\ as the   deformations of $X$ which lift to deformations of the resolution $\hX$. If $\pi\colon \hX \to X$ is a good crepant resolution, then $\mathbf{Def}_{\hX}$ is unobstructed  by a theorem of Gross; \textit{cf.} \cite[Proposition 3.4]{gross_defCY}. If $\pi'\colon X' \to X$ is a small resolution, then $\mathbf{Def}_{X'}$ is unobstructed by \cite[Proposition 2.1]{F}. If $X$ is a local complete intersection singularity, then $\mathbf{Def}_X$ is always unobstructed (\textit{cf.} \cite[Section~6]{Looijenga}).
 
In trying to understand the image of the morphism $\mathbf{Def}_{\hX} \to \mathbf{Def}_X$, and more concretely the image of $H^1(\hX;T_{\hX})$ in $H^0(X; T^1_X)$, and their relevance in the Calabi--Yau case, there are two major obstacles: 

\begin{enumerate}[(1)]
\item In the local setting,  the image of $H^1(\hX;T_{\hX})$ in $H^0(X; T^1_X)$ is not a birational invariant, i.e.\ is not independent of the choice of a good resolution. The possible naturally occurring birationally invariant subspaces of $H^0(X; T^1_X)$ are rather the images of $H^1(\hX; \Omega^{n-1}_{\hX}(\log E))$ or $H^1(\hX; \Omega^{n-1}_{\hX}(\log E)(-E))$. These images are studied in \cite[Theorem 2.1(iii)]{FL}, where we prove that the image of $H^1(\hX; \Omega^{n-1}_{\hX})$ is the same as that of $H^1(\hX; \Omega^{n-1}_{\hX}(\log E)(-E))$, at least in the local complete intersection case or if $\dim X =3$, and is thus independent of the choice of resolution. However, this image  does not seem to have an obvious deformation-theoretic interpretation. One important case where such an interpretation exists is when $\hX$ is a good crepant resolution of $X$:  In this case $\Omega^{n-1}_{\hX} \cong T_{\hX}$, and hence the image of $H^1(\hX; \Omega^{n-1}_{\hX}(\log E)(-E))$ agrees with that of $H^1(\hX;T_{\hX})$.  A similar result holds in the case of a small resolution (Proposition~\ref{smallvbig}\eqref{smallvbig-1}). 
\item In the global setting, where $Y$ is a Calabi--Yau variety, it seems difficult to lift deformations arising in this manner to global deformations of $Y$. For example, in dimension $3$ and for small resolutions $Y'$, this issue is connected with Clemens' conjecture about smooth rational curves in $Y'$ in the case where $Y'$ is a quintic threefold, see \cite{Clemensconj}, which is still open and where we have nothing new to add. For another example,   the deformations of a quintic threefold $Y$ with an $O_{16}$ singularity $x\in Y$ are versal for the deformations of the isolated singularity at $x$; \textit{i.e.} the map $\mathbb{T}^1_Y \to H^0(Y; T^1_Y)$ is surjective. Hence, if $\hY \to Y$ is the natural crepant resolution and $X$ is a good Stein representative of the germ $(Y,x)$ with crepant resolution $\hX$, the image of  $\mathbb{T}^1_Y \to H^0(Y; T^1_Y)$ contains the image of $H^1(\hX;T_{\hX})$. However, the analogous result fails for hypersurfaces of degree $n+2$ in $\Pee^{n+1}$ containing an isolated singularity isomorphic to the cone over a hypersurface of degree $n$ in $\Pee^n$ for $n \ge 4$, and in this case ``most'' of  the image of $H^1(\hX;T_{\hX})$ fails to lift to $\mathbb{T}^1_Y$.  Thus, the strategy of \cite{FL} is to work modulo the deformations induced from a resolution, i.e.\  with $K=\Coker\left(H^1(\hX;\Omega^{n-1}_{\hX}(\log E)(-E))\to H^0(X; T^1_X)\right)\cong H^2_E(\hX;\Omega^{n-1}_{\hX}(\log E))$  (see \cite[Theorem 2.1(v)]{FL}). This has the virtue of globalizing well in the Calabi--Yau and Fano case; see \cite[Sections~4 and 5]{FL}.
\end{enumerate}

Nonetheless, the study of $\mathbf{Def}_{\hX}$ and of $H^1(\hX;T_{\hX})$ involves  a lot of interesting geometry, and the goal of this paper is to investigate some of this geometry. For the reasons outlined above, we restrict to the local case and (mostly) either to the good crepant case or to the case of a small resolution. 
 Finally, for most of our results, we restrict to dimension $n=3$.  For a general $n\ge 3$,   by  a theorem of Reid \cite[Theorem 2.6]{Reidcanon},   a  canonical Gorenstein singularity, not necessarily admitting a crepant resolution,  can be realized as the total space of a one-parameter deformation of a Gorenstein rational   or   elliptic singularity of dimension $n-1$. In dimension $2$, Gorenstein rational  singularities are rational double points, and this case does not arise if $\hX$ is a good crepant resolution. As for Gorenstein elliptic (minimally elliptic) singularities,  the most well-studied classes of such singularities are the simple elliptic and cusp singularities, and their deformation theory has been studied extensively. In particular, as discussed below, semistable models for deformations of cusp singularities are a plentiful source of examples of good crepant resolutions and can be obtained systematically by the methods of   \cite{FriedmanMiranda,Engel,FriedmanEngel}.

The contents of this paper are as follows.   Section~\ref{Section1} analyzes the deformation theory of good  crepant resolutions $\pi\colon\hX \to X$, where $\dim X \ge 3$. The main result is Theorem~\ref{goodcrepdef},  which gives some very general results about the first-order deformations of $\hX$ and their relation to first-order deformations of $E$. In this case, the tangent space to the deformation functor $\mathbf{Def}_E$ is the vector space $\mathbb{T}^1_E$, and there is a corresponding sheaf $T^1_E$. Then 
Theorem~\ref{goodcrepdef} relates   the tangent space to $\mathbf{Def}_{\hX}$ and the corresponding obstruction space  to the cohomology of the exceptional divisor $E$  and its components $E_i$.  Among other things, we show the following. 

\begin{theorem} Let $\pi\colon \hX \to X$ be a good crepant resolution of the isolated rational singularity $X$, with $n=\dim X \geq 3$, and let $E = \bigcup_{i=1}^rE_i$ be the exceptional divisor of\, $\pi$.  Then the maps $\mathbb{T}^1_E \to H^0(E;T^1_E)$ and $H^1(\hX;T_{\hX}) \to H^0(E;T^1_E)$ are surjective. In particular, all first-order smoothing directions for $E$ are realized via first-order deformations of\, $\hX$. The   map $H^1(\hX;T_{\hX}) \to \mathbb{T}^1_E$ is surjective if and only if     $H^2(\hX; T_{\hX}(- E)) =0$. In this case, all first-order locally trivial deformations of\, $E$ are realized by first-order deformations of\,  $\hX$ and the divisors $E_1, \dots, E_r$.  
\end{theorem} 

More precise results require   better control of the structure of $E$, which in turn leads  to restricting to the case $\dim X =3$,  the running assumption starting with Section~\ref{Section2}. 

In dimension $2$,   Gorenstein canonical singularities are  du Val singularities, also called rational double point (RDP),   simple, or ADE  singularities. Their structure is well understood, as is their deformation theory.  The purpose of Sections~\ref{Section2} and~\ref{Section3} is to give a partial classification of the isolated Gorenstein canonical threefold singularities $(X,x)$ which admit good crepant resolutions,   $3$-dimensional analogues of the ADE case, and to discuss their associated deformation theory.   As noted above, there is a close connection between isolated threefold canonical Gorenstein singularities  and one-parameter smoothings of minimally elliptic  singularities. In the case of simple elliptic and cusp singularities, such smoothings are in turn   closely related to  degenerations  
of $K3$ surfaces (\textit{cf.} for example  \cite{FriedmanMiranda}). This leads us to define  \textsl{divisors of Type II, Type III${}_1$, and Type III${}_2$} (Definition \ref{defcrep1} and Figures \ref{fig1}, \ref{fig2}, \ref{fig3}). We then obtain a partial classification of the threefold singularities admitting good crepant resolutions in the special case of the total space of a one-parameter smoothing of a simple elliptic or cusp singularity, as follows. 

\begin{theorem}[= Theorems \ref{somecrepresults} and \ref{partialconverse}]\label{0.1}  Let $(X,x)$ be an isolated Gorenstein canonical  singularity  of dimension $3$ with a good crepant resolution $\pi \colon \hX \to X$, and let $E =\pi^{-1}(x)$ be the reduced exceptional divisor. 
  \begin{enumerate}
  \item[\rm(i)] If the general hypersurface section of\, $X$ passing through $x$ is a simple elliptic singularity, then $E$ is of Type II.
\item[\rm(ii)] If the general hypersurface section of\, $X$ passing through $x$ is a cusp and $\omega_E^{-1}$ is nef and big, then $E$ is of Type III${}_1$ or Type III${}_2$.
 \item[\rm(ii)$'$] If the general hypersurface section $S$ of\, $X$ passing through $x$ is a cusp and the full inverse image  $\pi^{-1}(S)$ has normal crossings, then after a sequence of flops $($elementary modifications of type 2\,$)$, $\omega_E^{-1}$ becomes nef and big, hence $E$ is of Type III${}_1$ or Type III${}_2$.  
 \end{enumerate}
 \end{theorem}
 
Combining Theorem~\ref{0.1}  with Theorem~\ref{goodcrepdef}, we are able   to obtain a deeper understanding of   various deformation-theoretic invariants, especially in the Type II and Type III${}_1$ cases. In particular, the following is a somewhat less precise formulation of  Proposition~\ref{prop6}, Proposition~\ref{prop2.12TypeIII1}, and Remark~\ref{remark2.12TypeIII2}. 

\begin{theorem} If $E$ is of Type II and irreducible or if $E$ is of Type III${}_1$, then $H^2(\hX; T_{\hX}(-E)) =0$ and hence the  map $H^1(\hX;T_{\hX}) \to \mathbb{T}^1_E$ is surjective. However, if $E$ is of Type II and reducible, then the  map $H^1(\hX;T_{\hX}) \to \mathbb{T}^1_E$ is never surjective, and  if $E$ is of Type III${}_2$, then the  map $H^1(\hX;T_{\hX}) \to \mathbb{T}^1_E$ is not in general surjective. 
\end{theorem}

In Section~\ref{Section4}, we switch our attention to the case of singularities which admit small resolutions $p \colon X'\to X$,     technically  a much simpler case. Here, we  relate $H^1(X'; T_{X'})$ to the birational invariants $H^1(\hX;\Omega^2_{\hX}(\log E)(-E))$,  $H^1(\hX;\Omega^2_{\hX})$, and $H^1(\hX;\Omega^2_{\hX}(\log E))$ arising from a good resolution; these invariants are controlled by the Du Bois invariants $b^{p,q}$ and link invariants $\ell^{p,q}$ introduced by Steenbrink; see \cite{SteenbrinkDB}. In particular, we recover some results of Steenbrink regarding the dimension of the versal deformation spaces for such singularities  (Remark \ref{rem-dim})  and  discuss some interesting examples (Examples \ref{ex-a2n-1} and \ref{ex-cA}). After the first version of this paper was posted, Sz-Sheng Wang sent us a preprint (now \cite{SSWang}) which has substantial overlap with the material in Section~\ref{Section4}.  

In the final Section~\ref{Section5}, we discuss  a noncrepant example of a very special type, the blowup $\hX$ of a smooth curve $C$ which is the exceptional set of a small resolution $\pi \colon X'\to X$. In this case, $X$ is a threefold $A_{2n-1}$ singularity, \textit{i.e.} defined locally by the equation $x^2 + y^2 + z^2 + w^{2n}$,  where we assume $n \ge 2$. The question here is to relate the deformations of $\hX$ to those of $X'$ and $X$. In particular, we show the following (Theorem~\ref{noncrepthm}). 

\begin{theorem} For the above example, let $(S_{\hX},0)$ and $(S_{X'},0)$ be the germs prorepresenting the functors $\mathbf{Def}_{\hX}$ and $ \mathbf{Def}_{X'}$, respectively. Then the induced morphism $S_{\hX}\to S_{X'}$ is finite of degree $n$, and its differential at the origin has a $1$-dimensional kernel and cokernel.  
 \end{theorem}

This kind of  example is also relevant to the study of deformations of $\Q$-factorial terminal threefold singularities  such as the $A_{2n}$ singularities in dimension $3$. While this example is very specific, it helps to illustrate the difference between the image of $H^1(\hX;T_{\hX})$ and  the birationally invariant image of $H^1(X'; T_{X'})$.  It  is also interesting  from the perspective of the minimal model program. Generally speaking,  the analysis of this paper shadows the steps of the minimal program. Namely, Sections~\ref{Section1}--\ref{Section3} roughly parallel the fact that a canonical threefold singularity  has a partial crepant (divisorial) resolution with terminal singularities (\textit{cf.} \cite[Theorem 6.23]{KollarMori}). Similarly, Section~\ref{Section4} is the deformation-theoretic counterpart of the statement that a terminal singularity admits a small partial resolution to a terminal $\Q$-factorial singularity (\textit{cf.} \cite[Theorem 6.25]{KollarMori}), which cannot be further improved. There is however an important difference between the deformation-theoretic point of view and that of the minimal model program:   In our arguments we need the partial resolutions to be actual resolutions; \textit{i.e.} we only consider    crepant partial resolutions of a canonical singularity which are smooth,   not just terminal. Nonetheless, we believe that the discussion here captures some important general phenomena for these classes of singularities, which in turn will help to better understand the geometry of the moduli spaces of Calabi--Yau varieties, especially in dimension $3$. 

\subsection*{Notation and conventions} Throughout the paper, we work with the notation and assumptions made above:  $(X,x)$ denotes  the germ of an isolated  singularity,  $X$   is  a \textsl{good Stein representative} for the germ $(X,x)$, \textit{i.e.} a contractible Stein representative with a unique singular point $x$   (\textit{cf.}   \cite[Section~2]{Looijenga}), and $\pi\colon \hX\to X$ is a  \textsl{good resolution}; \textit{i.e.}  $\pi^{-1}(x) = E$ is a divisor with simple normal crossings. Unless otherwise specified, all singular cohomology (including local cohomology) is with $\Cee$-coefficients.  

\subsection*{Acknowledgement} 
We would   like to thank the referee for a careful reading of our paper, and for extensive comments which  helped us to improve it. We would also like to thank Paul Hacking for pointing out an error in a previous version of this paper.

\section{Deformation theory in the good crepant case}\label{Section1}

We begin with a general definition. 
 
\begin{definition}\label{defequi}  Let $\pi \colon \hX \to X$ be a resolution of   $X$. Then $\pi$ is \textsl{equivariant} if $R^0\pi_*T_{\hX}\cong T^0_X$. By \textit{e.g.} \cite[Lemma 3.1]{F}, a small resolution is equivariant. Note that, for $q> 0$, $R^q\pi_*T_{\hX}$ is a torsion sheaf supported on $x$.

The resolution $\pi \colon \hX \to X$ is
\begin{enumerate}[(1)]
\item a \textsl{good resolution}  (sometimes called a  \textsl{log resolution})  if $\pi^{-1}(x) = E$ is a divisor with simple normal crossings;  
\item  a \textsl{good equivariant resolution} if $\pi$ is good and equivariant (by resolution of singularities, good equivariant resolutions always exist); 
\item  \textsl{crepant} if $K_{\hX} = \pi^*\omega_X$ and hence $K_{\hX}\cong \scrO_{\hX}$, and is a \textsl{good crepant resolution} if $\pi$ is also a good resolution.  Thus, with these definitions, a small resolution is crepant but not a good crepant resolution.  (Note: If the isolated normal singularity $(X,x)$ admits a not necessarily good resolution $\pi\colon \hX \to X$ with $K_{\hX} \cong \scrO_{\hX}$, then $(X,x)$ is automatically a rational Gorenstein singularity; \textit{cf.} \cite[Corollary~11.9]{Kollar}). 
\end{enumerate}
\end{definition}

Given a good resolution $\pi\colon \hX \to X$,  let $E = \pi^{-1}(x) = \bigcup_{i=1}^rE_i$, where the $E_i$ are smooth divisors in~$\hX$.

\begin{proposition}\label{conj1} A crepant resolution is equivariant.
\end{proposition}
\begin{proof}
We begin by showing the following. 

\begin{lemma}\label{conj2} Let $\pi \colon \hX \to X$  be a crepant resolution, and let $\pi'\colon X' \to X$ be an arbitrary resolution. Then there exist a closed analytic subset $\widehat{V}$ of\, $\hX$ of codimension at least $2$ and  a proper analytic subset $V'$ of\, $X'$ such that the birational map $X'\dasharrow \hX$ restricts on $X'-V'$ to a surjective  morphism $\nu\colon X'-V' \to \hX - \widehat{V}$.
\end{lemma}
\begin{proof} By Hironaka's theorem, there is a blowup $f\colon \widetilde{X}\to\hX$ which dominates $X'$. We can further assume that all centers of blowups lie over the inverse image of the singular point $x$.  Let $\widehat{V}$ be the image in $\hX$ of the centers of the blowups; hence $f$ is an isomorphism from $\widetilde{X} -f^{-1}(\widehat{V})$ to $\hX -V$.    Moreover, $K_{\widetilde{X}} = \scrO_{\widetilde{X}}(G)$, where $G = \sum_in_iG_i$ is a divisor with $n_i > 0$ for all $i$ such that $f(G) \subseteq \widehat{V}$. Since $K_{X'}\cong \scrO_{X'}(D)$ for an effective divisor $D$ whose image in $X$ is the point $x$, it follows easily that all of the exceptional divisors for the morphism $\widetilde{X} \to X'$ are of the form $G_i$ for some $i$. Thus, if $V'$ is the closure of the union of the images of the $G_i$ which are exceptional for the morphism $\widetilde{X} \to X'$, there is a surjective morphism $X'-V' \to \hX - \widehat{V}$.
\end{proof}

\begin{remark}\label{crepremark} The argument of Lemma~\ref{conj2} proves the standard fact that if in addition $X'$ is also a crepant resolution of $X$, then $X'$ and $\hX$ are isomorphic outside a set of  codimension $2$.
\end{remark}

Continuing with the proof of Proposition~\ref{conj1}, we must show that the natural injective map $R^0\pi_*T_{\hX}\to T^0_X$ is surjective. Choose an equivariant resolution $\pi'\colon X'\to X$. If $\xi$ is a local section of $T^0_X$, then $\xi$ lifts to a section of $T_{X'}$ over the inverse image of the open set and thus defines a section of  $T_{\hX}$ over the complement of $V$ in the notation of Lemma~\ref{conj2}. By Hartogs, this section extends to a section $\hat\xi$ of $T_{\hX}$. The image of $\hat\xi$ in $T^0_X$ is then $\xi$. Thus $R^0\pi_*T_{\hX}\to T^0_X$ is surjective and therefore an isomorphism.
\end{proof}

Now suppose that  $\hX$ is   a good, not necessarily crepant resolution of  $X$. There is an exact sequence
$$0 \lra T_{\hX}(-\log E) \lra T_{\hX} \lra \bigoplus_iN_{E_i/\hX} \lra 0.$$
Here $H^1(\hX;T_{\hX}(-\log E))$ is the tangent space to the functor of deformations of $\hX$ keeping the divisors $E_i$. Let  $\mathcal{D}^\bullet$ be  the complex given by
$$T_{\hX} \lra N_{E/\hX}$$
(in degrees $0$, $1$, respectively). Then $\mathbb{H}^1(\hX; \mathcal{D}^\bullet)$ is the tangent space to deformations of $\hX$ keeping the divisor $E$ (as an effective Cartier divisor). Also, let $\mathcal{C}^\bullet$ be the complex
$$T_{\hX}|E \lra N_{E/\hX}.$$
Thus $\mathcal{C}^\bullet$ is the dual complex to the complex $I_E/I_E^2\to \Omega^1_{\hX}|E$, which is quasi-isomorphic to $\Omega^1_E$,  the sheaf of K\"ahler differentials on $E$. It follows that $\mathcal{H}^i(\mathcal{C}^\bullet)= T^i_E$, $i=0,1$, and $\mathbb{H}^i(E; \mathcal{C}^\bullet)=\Ext^i(\Omega^1_E, \scrO_E) = \mathbb{T}^i_E$.  There is a  commutative diagram
$$\begin{CD}
@. 0 @. 0 @. @. @.\\
@. @VVV @VVV @. @. @.\\
@. T_{\hX}(-E) @= T_{\hX}(-E) @. @. @.\\
@. @VVV @VVV @. @. @.\\
0@>>> T_{\hX}(-\log E) @>>> T_{\hX} @>>>\bigoplus_iN_{E_i/\hX} @>>> 0 @.\\
@. @VVV @VVV @VVV @. @.\\
0@>>> T^0_E @>>> T_{\hX}|E @>>>N_{E/\hX} @>>> T^1_E @>>> 0\\
@. @VVV @VVV @. @. @.\\
@. 0 @. 0\rlap{.} @. @. @.
\end{CD}$$
In particular, there is always an exact sequence
$$0 \lra \bigoplus_iN_{E_i/\hX} \lra N_{E/\hX} \lra  T^1_E\lra 0.$$
Moreover,  $\mathcal{H}^0(\mathcal{D}^\bullet)= T_{\hX}(-\log E)$ and  $\mathcal{H}^1(\mathcal{D}^\bullet)= \mathcal{H}^1(\mathcal{C}^\bullet)= T^1_E$.
Also, from the hypercohomology spectral sequences, there are exact sequences
\begin{gather*}
H^0\left(\hX;T_{\hX}\right) \lra H^0\left(E; N_{E/\hX}\right) \lra \mathbb{H}^1\left(\hX; \mathcal{D}^\bullet\right) \lra H^1\left(\hX;T_{\hX}\right) \lra H^1\left(E; N_{E/\hX}\right), \\
H^1\left(\hX; T_{\hX}\left(-\log E\right)\right) \lra \mathbb{H}^1\left(\hX; \mathcal{D}^\bullet\right) \lra H^0\left(E;T^1_E\right) \lra H^2\left(\hX; T_{\hX}\left(-\log E\right)\right) 
\end{gather*}
as well as the usual exact sequences
\begin{gather*}
H^0\left(E;T_{\hX}|E\right) \lra H^0\left(E; N_{E/\hX}\right) \lra \mathbb{T}^1_E \lra H^1\left(E;T_{\hX}|E\right) \lra H^1\left(E; N_{E/\hX}\right), \\
0\lra H^1\left(E; T^0_E\right) \lra \mathbb{T}^1_E \lra H^0\left(E;T^1_E\right) \lra H^2\left(E; T^0_E\right) \lra \mathbb{T}^2_E \lra H^1\left(E;T^1_E\right) \lra 0.
\end{gather*}
Most of the homomorphisms in the above exact sequences have a geometric meaning.

Now suppose $\hX$ is a good crepant resolution of   $(X,x)$. Then   $T_{\hX}(-\log E)$ is isomorphic to $\Omega^{n-1}_{\hX}(\log E) (-E)$, so, by the vanishing theorem of  Guill\'en, Navarro Aznar, Pascual Gainza, Puerta, and Steenbrink (see \textit{e.g.}  \cite[Proof of Proposition~7.30 (b$’$), p.~181]{PS}), for $p\geq 2$, 
$$H^p\left(\hX; T_{\hX}(-\log E)\right)  = H^p\left(\hX; \Omega^{n-1}_{\hX}(\log E) (-E)\right)=0.$$ 
Also,  $N_{E_i/\hX} =\omega_{E_i} = K_{E_i}$ and $N_{E/\hX} =\omega_E$. Moreover, as previously noted, $\mathbf{Def}_{\hX}$ is unobstructed.

Finally, we have the following result of Steenbrink; see \cite[Lemma 2.14]{Steenbrink}. 

\begin{theorem}[Steenbrink]\label{Steenlemma} In the above notation, if $X$ has an isolated rational singularity, then $H^i(E; \scrO_E) =0$ for $i > 0$. 
\end{theorem}

With these preliminaries, we turn now to the main result of this section. 

\begin{theorem}\label{goodcrepdef} Let $\pi\colon \hX \to X$ be a good crepant resolution of the isolated rational singularity $X$, with $n=\dim X \geq 3$.
\begin{enumerate}
\item\label{goodcrepdef-1} We have $\mathbb{H}^1(\hX; \mathcal{D}^\bullet) \cong H^1(\hX;T_{\hX})$. In other words, the first-order deformations of\, $\hX$ are exactly the first-order deformations keeping the effective divisor $E$.
\item\label{goodcrepdef-2} We have $H^0(E;T^1_E) \cong \bigoplus_iH^1(E_i;N_{E_i/\hX}) = \bigoplus_iH^1(E_i; \omega_{E_i})$. Thus $H^0(E;T^1_E) = 0$ if and only if   $h^{0,n-2}(E_i) = 0$ for every $i$. $($The condition  $H^0(E;T^1_E)=0$ is the condition that all first-order deformations of\, $\hX$ induce locally trivial first-order deformations of\, $E$.$)$ More generally, $\dim H^0(E;T^1_E) =\sum_ih^{0,n-2}(E_i)$.
\item\label{goodcrepdef-3} If $\dim X = 3$, then $H^1(E;T^1_E)$ has dimension $r-1$, where $r$ is the number of components of\, $E$, and is more intrinsically dual to the cokernel of\,  $ H^0(E;\scrO_E) \to \bigoplus_iH^0(E_i; \scrO_{E_i})$. If $\dim X > 3$, then $\dim H^1(E;T^1_E) =\sum_i\dim H^{n-3}(E_i; \scrO_{E_i})$.
\item\label{goodcrepdef-4} We have $H^2(\hX; T_{\hX}) \cong \bigoplus_iH^2(E_i;N_{E_i/\hX})$, and if $H^3(\hX; T_{\hX}(-E) ) =0$ $($which is always satisfied if $\dim X =3)$,   then $H^2(E; T^0_E) = 0$, \textit{i.e.} all locally trivial first-order deformations of\, $E$ are unobstructed. Thus, in this case, 
$$\dim \mathbb{T}^2_E =\dim  H^1\left(E;T^1_E\right) =\begin{cases} r-1 &\text{if $\dim X = 3$},\\
\sum_i\dim H^{n-3}(E_i; \scrO_{E_i}) &\text{if $\dim X > 3$}.
\end{cases}$$
\item\label{goodcrepdef-5} The maps $\mathbb{T}^1_E \to H^0(E;T^1_E)$ and $H^1(\hX;T_{\hX}) \to H^0(E;T^1_E)$ are surjective. In particular, all first-order smoothing directions for $E$ are realized via first-order deformations of\, $\hX$.
\item\label{goodcrepdef-6} The  map $H^1(\hX;T_{\hX}) \to \mathbb{T}^1_E$ is surjective if and only if  $H^1(\hX; T_{\hX}(-\log E)) \to H^1(E;T^0_E)$ is surjective, which holds if and only if $H^2(\hX; T_{\hX}(- E)) =0$. In this case, all first-order locally trivial deformations of\, $E$ are realized by first-order deformations of\,  $\hX$ and the divisors $E_1, \dots, E_r$.  
\end{enumerate}
\end{theorem}
\begin{proof} \eqref{goodcrepdef-1}~ By adjunction and Serre duality, $H^0(E; N_{E/\hX}) =H^0(E; \omega_E)$  is dual to $H^{n-1}(E;\scrO_E) = 0$. Likewise, $H^1(E; N_{E/\hX}) =H^1(E; \omega_E)$  is dual to $H^{n-2}(E;\scrO_E) = 0$. Hence $\mathbb{H}^1(\hX; \mathcal{D}^\bullet) \cong H^1(\hX;T_{\hX})$.

\smallskip
\eqref{goodcrepdef-2}~ There is an exact sequence
$$H^0\left(E; N_{E/\hX}\right) \lra H^0\left(E;T^1_E\right) \lra \bigoplus_iH^1\left(E_i;N_{E_i/\hX}\right) \lra H^1\left(E; N_{E/\hX}\right).$$
As in \eqref{goodcrepdef-1},  $H^0(E; N_{E/\hX}) =   H^1(E; N_{E/\hX}) =0$. Thus  $H^0(E;T^1_E) \cong \bigoplus_iH^1(E_i; N_{E_i/\hX}) $. But $H^1(E_i; N_{E_i/\hX}) = H^1(E_i;\omega_{E_i})$ is Serre dual to $H^{n-2}(E_i;\scrO_{E_i})$  and therefore  has dimension $h^{0,n-2}(E_i)$. In particular, $H^0(E;T^1_E) = 0$ if and only if  $h^{0,n-2}(E_i) =0$ for all $i$.  

\smallskip
\eqref{goodcrepdef-3}~ Continuing with the above exact sequence, we have
$$0= H^1\left(E; N_{E/\hX}\right) \lra H^1\left(E;T^1_E\right) \lra \bigoplus_iH^2\left(E_i; N_{E_i/\hX}\right) \lra H^2\left(E; N_{E/\hX}\right) \lra 0.$$
If $\dim X = 3$, then   $ H^1(E;T^1_E)$ is dual to the cokernel of $H^0(E; \scrO_E) \to \bigoplus_iH^0(E_i; \scrO_{E_i}) $ and hence has dimension $r-1$. If $\dim X > 3$, then $H^2(E_i; N_{E_i/\hX}) = H^2(E_i;\omega_{E_i})$ is Serre dual to $H^{n-3}(E_i;\scrO_{E_i})$ and $H^2(E; N_{E/\hX})$ is Serre dual to $H^{n-3}(E; \scrO_E) = 0$. Thus $\dim H^1(E;T^1_E) =\sum_i\dim H^{n-3}(E_i; \scrO_{E_i})$ if $\dim X > 3$.

\smallskip
\eqref{goodcrepdef-4}~ The statement about $H^2(\hX; T_{\hX})$ follows from the exact sequence
$$H^2\left(\hX; T_{\hX}(-\log E)\right)= 0 \lra H^2\left(\hX; T_{\hX}\right) \lra \bigoplus_iH^2\left(E_i; N_{E_i/\hX}\right) \lra H^3\left(\hX; T_{\hX}(-\log E)\right)= 0 .$$
The long exact sequence associated to
$$0\lra T_{\hX}(-E) \lra T_{\hX}(-\log E) \lra T^0_E \lra 0$$
gives rise to an exact sequence
$$H^2\left(\hX; T_{\hX}(-\log E)\right) \lra H^2\left(E; T^0_E\right) \to H^3\left(\hX; T_{\hX}(-E) \right).$$
As we have seen, $H^2(\hX; T_{\hX}(-\log E))=0$, and $H^3(\hX; T_{\hX}(-E) ) =0$ if $\dim X =3$ for dimension reasons. Thus $H^2(E; T^0_E) = 0$.

\smallskip
\eqref{goodcrepdef-5}~ There is a commutative diagram
$$\begin{CD}
0 @>>> H^1\left(\hX; T_{\hX}(-\log E)\right)  @>>> H^1\left(\hX; T_{\hX}\right) @>>>  \bigoplus_iH^1\left(N_{E_i/\hX}\right) @>>> 0\\
@. @VVV @VVV @VV{\cong}V @.\\
0 @>>> H^1\left(E; T^0_E\right) @>>> \mathbb{T}^1_E @>>> H^0\left(E;T^1_E\right) @>>> H^2\left(E;T^0_E\right).
\end{CD}$$
(Here, the top right  arrow is surjective because $H^2(\hX; T_{\hX}(-\log E))=0$.) Thus the induced map $H^1(\hX;T_{\hX}) \to H^0(E;T^1_E)$ is surjective, and therefore the map $\mathbb{T}^1_E \to H^0(E;T^1_E)$ is surjective as well.

\smallskip
\eqref{goodcrepdef-6}~  From the diagram in \eqref{goodcrepdef-5}, the   map $H^1(\hX;T_{\hX}) \to \mathbb{T}^1_E$ is surjective if and only if the map  $H^1(\hX; T_{\hX}(-\log E)) \to H^1(E;T^0_E)$ is surjective.  Since  $H^2(\hX; T_{\hX}(- \log E))=0$, the cokernel of   $H^1(\hX; T_{\hX}(-\log E)) \to H^1(E; T^0_E)$ is $H^2(\hX; T_{\hX}(- E))$. Thus $H^1(\hX;T_{\hX}) \to \mathbb{T}^1_E$ is surjective if and only if $H^2(\hX; T_{\hX}(- E))=0$.
\end{proof} 

\begin{remark}\label{rem-1.7}\leavevmode
  \begin{enumerate}[wide]
    \item  In the situation of \eqref{goodcrepdef-5}, it follows from  \cite[Theorem 2.1(iii)]{FL} that, for a  crepant isolated rational singularity,
$$H^1\left(\hX; T_{\hX}\right) \cong H^1\left(\hX;\Omega^{n-1}_{\hX}\right) \cong H^1\left(\hX;\Omega^{n-1}_{\hX}(\log E)(-E)\right)\oplus H^1_E\left(\hX;\Omega^{n-1}_{\hX}\right).$$
Moreover, $H^1(\hX;\Omega^{n-1}_{\hX}(\log E)(-E)) \cong H^1(\hX; T_{\hX}(-\log E))$. The induced map
$$H^1_E\left(\hX;\Omega^{n-1}_{\hX}\right) \lra H^1\left(\hX;\Omega^{n-1}_{\hX}\right) \lra H^1\left(E;\Omega^{n-1}_E/\tau^{n-1}_E\right) \cong \bigoplus _iH^1\left(E_i; \omega_{E_i}\right),$$
is an isomorphism, and thus by \eqref{goodcrepdef-5} there is a splitting
$$H^1\left(\hX; T_{\hX}\right) \cong H^1\left(\hX; T_{\hX}(-\log E)\right) \oplus H^0\left(E; T^1_E\right).$$
Note that, while $H^1(\hX; T_{\hX}(-\log E))$ and $H^1(\hX; T_{\hX})$ have the same image in $H^0(X; T^1_X)$, this is just a statement about the differential of the corresponding morphism of deformation functors $\mathbf{Def}_{\hX}\to \mathbf{Def}_X$, and it is reasonable to ask if the actual morphism of deformation functors is finite (meaning that the corresponding morphism  of the analytic germs which prorepresent them is finite).
 
\item If $\dim X =3$ and $E$ is smooth,  we will show in the next section that $H^1(\hX;T_{\hX}) \to \mathbb{T}^1_E$ is surjective.   However, it is not in general an isomorphism, for example in case  $X$ is an $O_{16}$ singularity, the cone over a smooth cubic surface $E$. In this case, a  calculation shows that $\dim H^0(X;T^1_X) = 16$ and that $\ell = 6$, where $\ell=\ell^{1,1}$ is the link invariant  of \cite{SteenbrinkDB}. By \cite{Wahl},  $H^0(X; R^1\pi_*T_{\hX})\cong H^1(\hX; \Omega^2_{\hX})$ is   the nonnegative part of the deformations of $O_{16}$ and is easily computed to have dimension $5$. 
But $\dim H^1(E; T_E) = 4$, so that $H^0(X; R^1\pi_*T_{\hX}) \to   H^1(E; T_E)$ is not injective. Here, the   weight zero piece of $H^0(X;T^1_X)$ has dimension $4$ and corresponds to deformations of the cubic surface,  hence gives all of the first-order deformations of $E$. However,  there is also a weight $1$ piece of dimension $1$. Starting with the cone over the Fermat cubic surface $E$, for example, the general weight $1$ deformation is given by $f_t = x^3 + y^3 + z^3 + w^3 + txyzw$, defining the singularity $X_t$. A calculation shows that, for $t\neq 0$,  $\dim H^0(X;T^1_{X_t}) = 15$, and hence that the invariant $a$ of \cite[Theorem 2.1(iv)]{FL} is not $0$ in this case. In fact, $a=1$, and the spectral sequence with $E_1$ page $E_1^{p,q}=H^q_E(\hX_t; \Omega^p_{\hX_t}) \Rightarrow H^{p+q}_E(\hX_t)$ does not degenerate at $E_2$ for $t\neq 0$. 

Finally, if $E$ is not smooth, then $H^1(\hX;T_{\hX}) \to \mathbb{T}^1_E$ typically fails to be surjective. We will give a geometric explanation for this failure in the next section.
\end{enumerate}
\end{remark}

\section{The good crepant case: Some classes of examples}\label{Section2}

In this section,  $\pi\colon \hX \to X$ denotes a good crepant resolution of the  rational, Gorenstein, isolated  singularity $X$, with $\dim X = 3$. Denote the exceptional divisor by $E =\pi^{-1}(x) = \bigcup_{i\geq 1}E_i$. We fix the following notation: $D_{ij} = E_i\cap E_j$, $D_i= E_i\cap \left(\bigcup_{j\neq i}E_j\right)=  \bigcup_{j\neq i} D_{ij}$, and $D =  \bigcup_iD_i  = E_{\text{sing}}$. Our goal is to describe in more detail the case where $E$ looks like a $K$-trivial semistable degeneration of $K3$ surfaces with one component supporting a negative-definite anticanonical divisor removed, or more generally a $K$-trivial semistable resolution of a one-parameter smoothing of a simple elliptic or cusp singularity, minus the component containing the resolution of the germ of the singularity.  Note that by \cite{FriedmanMiranda, Engel, FriedmanEngel}, there is a procedure for generating all such examples. There is also an easy criterion for determining the multiplicity and hence for deciding when such examples are good crepant resolutions of complete intersection or even hypersurface singularities. 

 Before we begin our discussion, we record the following, which holds for an isolated rational singularity $X$ of dimension $3$, not necessarily crepant. 
 
 \begin{proposition}\label{3dimtop}  Let $\pi \colon \hX \to X$ be a good resolution of the isolated rational singularity $(X,x)$ of dimension~$3$, with exceptional divisor $E =\pi^{-1}(x) = \bigcup_{i\geq 1}E_i$, $D_{ij} = E_i\cap E_j$, and triple points $t_{ijk}$.
 \begin{enumerate}
 \item\label{3dimtop-1} For $i> 0$, $H^i(E; \scrO_E) = 0$ and $H^i(|\Gamma|) = 0$ for $i>0$, where $\Gamma$ is the dual complex of\, $E$.
 \item\label{3dimtop-2} The two homomorphisms
 $\bigoplus _iH^1(E_i; \scrO_{E_i})  \to \bigoplus _{i<j}H^1(D_{ij}; \scrO_{D_{ij}})$ and $\bigoplus _iH^0(E_i; \Omega^1_{E_i})   \to \bigoplus _{i<j}H^0(D_{ij}; \Omega^1_{D_{ij}})$ are isomorphisms. 
 \item\label{3dimtop-3} Let $\Omega^1_E$ be the sheaf of K\"ahler differentials and $\tau^1_E$ the subsheaf of torsion differentials. Then $H^0(E; \Omega^1_E/\tau^1_E)  =0$ and $H^1(E) =0$.
 \item\label{3dimtop-4} The map $\bigoplus_iH^2(E_i) \to \bigoplus_{i<j}H^2(D_{ij})$ is surjective, and its kernel is $H^2(E)$. Hence  $b_2(E) = \sum_{i=1}^rb_2(E_i) - \#\{\text{\rm{double curves}}\}$. 
 \item\label{3dimtop-5}  Let  $L$ be the link of the singularity $(X,x)$, and let $\ell $ be the dimension of\, $H^2(L)$ or equivalently $H^3(L)$. Then $H^2(L)$ is a pure Hodge structure and  $H^2(L) = H^{1,1}(L)$, so that $\ell = \ell^{1,1}$, where $\ell^{1,1}$ is the link invariant of\, \cite{SteenbrinkDB}. Finally,
 $$\ell = b_2(E)-r=\sum_ib_2(E_i) -  r- \#\{\text{\rm{double curves}}\} = \sum_i(b_2(E_i)-1) -   \#\{\text{\rm{double curves}}\}.$$ 
 \end{enumerate}
 \end{proposition}
 \begin{proof} We shall just sketch the proof. By Theorem~\ref{Steenlemma}, $H^i(E; \scrO_E) = 0$ for $i> 0$. The weight spectral sequence for $E$ degenerates at $E_2$, and hence so does the Mayer--Vietoris spectral sequence for $\scrO_E$. This degeneration, along with the fact that $H^i(E; \scrO_E) = 0$, $i=1,2$, implies that $H^i(|\Gamma|) = 0$ for $i>0$ and that $\bigoplus _iH^1(E_i; \scrO_{E_i})  \to \bigoplus _{i<j}H^1(D_{ij}; \scrO_{D_{ij}})$ is an isomorphism. Then $\bigoplus _iH^0(E_i; \Omega^1_{E_i})   \to \bigoplus _{i<j}H^0(D_{ij}; \Omega^1_{D_{ij}})$ is an isomorphism as well, as follows by taking complex conjugation. This proves \eqref{3dimtop-1} and \eqref{3dimtop-2}. There is an exact sequence 
 $$0 \lra \Omega^1_E/\tau^1_E \lra \bigoplus_i\Omega^1_{E_i} \lra \bigoplus_{i<j}\Omega^1_{D_{ij}} \lra 0.$$
 Thus, by \eqref{3dimtop-2},  $H^0(E; \Omega^1_E/\tau^1_E)  =0$. There is a spectral sequence converging to  $H^{p+q}(E)$ with $E_1^{p,q} = H^q(E; \Omega^p_E/\tau^p_E)$. As $H^0(E; \Omega^1_E/\tau^1_E)   = H^1(E;\scrO_E) =0$, $H^1(E) =0$, proving \eqref{3dimtop-3}. The semipurity theorem (see \textit{e.g.} \cite[Corollary 1.12]{Steenbrink}) implies that the mixed Hodge structure on $H^k(E)$ is pure for $k\geq 3$. An examination of the weight spectral sequence shows that this forces the map $\bigoplus_iH^2(E_i) \to \bigoplus_{i<j}H^2(D_{ij})$ to be surjective and identifies $H^2(E)$ with its kernel. Thus 
 $b_2(E) = \sum_{i=1}^rb_2(E_i) - \#\{\text{double curves}\}$, which is \eqref{3dimtop-4}. By the link exact sequence and semipurity (see \cite[Corollary 1.12]{Steenbrink} again), there is an exact sequence of mixed Hodge structures
 $$0 \lra H_E^2\left(\hX\right)\lra H^2(E) \lra H^2(L) \lra 0,$$
 where as in the statement of  \eqref{3dimtop-5},  $L$ is the link of the isolated singularity. 
 By duality, $\dim H_E^2(\hX) = \dim H^4(E) = \sum _iH^4(E_i) = r$. Thus
 $$\ell = \dim H^2(L) = b_2(E) - r= \sum_{i=1}^rb_2(E_i) - \#\{\text{double curves}\} -r,$$
 using \eqref{3dimtop-4}. The above exact sequence also shows that $H^2(L)$ is a pure Hodge structure and that $H^2(L) = H^{1,1}(L)$, hence $\ell =\ell^{1,1}$.  This establishes all of the statements of \eqref{3dimtop-5}. 
 \end{proof}

\begin{remark}  Kawamata \cite[Section~1, p.\ 97]{Kawamata1988} 
  introduced an invariant $\sigma(Y)$, the \textsl{defect} of $Y$, for a normal projective variety measuring the failure of $\Q$-factoriality for $Y$. For a rational singularity $(X,x)$ in dimension $3$, there is a local analogue $\sigma(x)=\dim H^2(\hX)/\sum_i\Cee[E_i]$ defined by Namikawa--Steenbrink; \textit{cf.} \cite[Equation~(3.9)]{NS}.
  It is immediate to see that $\sigma(x)=\ell$ via the link exact sequence
$$\begin{CD}
0 @>>> H^2_E(\hX) @>>> H^2(E) @>>>  H^2(L) @>>> 0 \\
@. @VV{\cong}V  @VV{\cong}V @. @.\\
@. \bigoplus_iH_4(E_i) @>>> H^2(\hX)\rlap{.} @. @.
\end{CD}$$
The relationship between the local and global invariants is as follows. If $Y$ is a compact complex threefold with isolated rational singularities, for each $y\in Y_{\text{sing}}$, we have the link $L_y$ and the local invariant $\ell_y =\dim H^2(L_y) = \dim H^3(L_y)$. Let $T$ be the kernel of the natural map $\bigoplus_{y\in Y_{\text{sing}}}H^3(L_y) \to H^4(Y)$, and let $\sigma(Y)$ be the dimension of the image of $\bigoplus_{y\in Y_{\text{sing}}}H^3(L_y) \to H^4(Y)$. Thus clearly 
$$\dim T = \sum_{y\in Y_{\text{sing}}}\ell_y - \sigma(Y).$$
A straightforward argument with Mayer--Vietoris and semipurity shows that
$$\sigma(Y) = b_4(Y) -b_2(Y).$$
Thus the defect measures the extent to which Poincar\'e duality fails on $Y$.
\end{remark}

For the remainder of this section, we return to the case where $X$ has a crepant resolution $\pi\colon \hX\to X$. Following Reid \cite[Definition 2.5]{Reidcanon}, we make the following definition. 

\begin{definition}\label{definegeneral}  Let $(X, x)$ be the germ of an isolated singularity. We say that   \textsl{Property P holds for a general $t\in \mathfrak{m}_x$, or for the general hypersurface section of\, $X$ defined by $t$}, if there exists a finite-dimensional subspace $V$ of $\mathfrak{m}_x$, mapping surjectively onto $\mathfrak{m}_x/\mathfrak{m}_x^2$, such that Property P holds for all $t$ in a nonempty Zariski open subset of $V$.
\end{definition}

With that said, we have the following,   due to  Reid; see  \cite[Theorem 2.6]{Reidcanon}, \cite[Lemma~5.30]{KollarMori}.

\begin{theorem} Let $(X,x)$ be an isolated rational Gorenstein singularity of dimension $3$, and let $S$ be a hypersurface section of\, $X$ passing through $x$ defined by a general $t\in \mathfrak{m}_x$ in the sense of Definition~\ref{definegeneral}.  Then $S$ has a du Val or a minimally elliptic  singularity.  \qed
\end{theorem} 

\begin{remark}\leavevmode
\begin{enumerate}[wide]
\item Since an isolated compound du Val singularity is  terminal, the du Val case is excluded if $X$ has a crepant resolution. 

\item In an earlier version of this paper, we incorrectly stated that $S$ has either a du Val, a simple elliptic, or a cusp  singularity. However, Paul Hacking communicated to us an example where the general hypersurface section has a Dolgachev (triangle) singularity. This example has a partial (singular) crepant resolution. It is possible that the existence of a good crepant resolution imposes additional constraints on general hypersurface sections.
  \end{enumerate}
\end{remark}

In the above situation, fix a general hypersurface section $S$ of $X$, and define $E_0$ to be the proper transform of $S$ on $\hX$. Thus $E_0\to S$ is birational, but the inverse image of the singular point is $\bigcup_{i\geq 1} (E_0\cap E_i)$. If $S$ is defined by the function $t\in \mathfrak{m}_{X,x}$, then as divisors $(\pi^*t) = E_0+\sum_{i\geq1}a_iE_i$, where $a_i$ is an integer $\geq 1$. Hence  
$$\scrO_{\hX}\left(E_0+\sum_{i\geq1}a_iE_i\right) \cong \scrO_{\hX}.$$
By the crepant assumption, $K_{E_i} =\scrO_{\hX}(E_i)|E_i = \scrO_{E_i}(E_i)$, including the case $i=0$ where we might have to replace $K_{E_0}$ by $\omega_{E_0}$, and $\omega_E = \scrO_{\hX}(E)|E = \scrO_{E}(E)$.

Motivated by the description of semistable smoothings of simple elliptic and cusp singularities (\textit{cf.} \cite{FriedmanMiranda} and Remark~\ref{Kulikov} below), we make the following definition.

\begin{definition}\label{defcrep1} A compact analytic surface $E = \bigcup_{i\geq 1}E_i$ with simple normal crossings is of \textsl{Type II} (Figure~\ref{fig1}) if, after relabeling, $E= \bigcup_{i=1}^rE_i$, where the following hold: 
  \begin{enumerate}
\item\label{defcrep1-1} $E_1, \dots, E_{r-1}$ are elliptic ruled (not necessarily minimally ruled),  and $E_r$ is rational, with $-K_{E_r}$ nef and big. (Here a smooth surface $S$ is a not necessarily minimally ruled surface over a base curve $D$ if there exists a morphism $S\to D$ whose generic fiber is $\Pee^1$, and it is elliptic ruled if $D$ is an elliptic curve.) 
\item\label{defcrep1-2}  The dual complex is a line segment with adjacent vertices $E_1$, \dots, $E_r$, or a single point if $r=1$. 
\item\label{defcrep1-3}  With $D_{i,i+1} = E_i\cap E_{i+1}$, $1\leq i\leq r-1$ as above, $D_{i,i+1}$ is a smooth elliptic curve.
\item\label{defcrep1-4}  There exists a smooth elliptic curve $D_0 \subseteq E_1$, disjoint from $D_{12}$, such that, in case $r> 1$, $K_{E_1} = \scrO_{E_1}(-D_0-D_{12})$, 
$K_{E_i} = \scrO_{E_i}(-D_{i-1,i}-D_{i,i+1})$ for $1\leq i\leq r-1$, and $K_{E_r} = \scrO_{E_r}(-D_{r-1,r})$, and   $K_{E_1} = \scrO_{E_1}(-D_0)$ in case $r=1$.
\item\label{defcrep1-5}  $\scrO_E(E) = \omega_E = \scrO_E(-D_0)$. 
\end{enumerate} 
\begin{figure}[htb]
\includegraphics[scale=0.7]{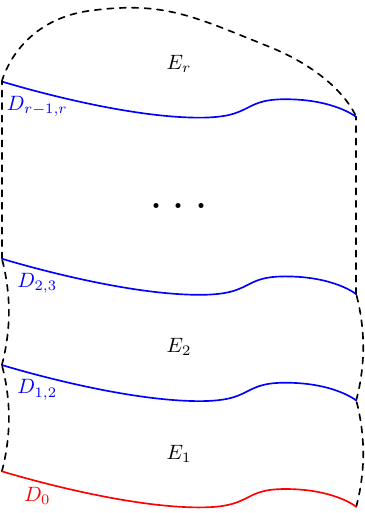}
\caption{Exceptional divisor $E$  of Type $II$}\label{fig1}
\end{figure}
Note that \eqref{defcrep1-5} implies that $\scrO_{D_{i,i+1}}(E) \cong \scrO_{D_{i,i+1}}$, and thus $T^1_E\cong \bigoplus_i\scrO_{\hX}(E)|D_{i,i+1}\cong \scrO_D$, where $D = \bigcup_{i=1}^{r-1}D_{i,i+1} = E_{\text{sing}}$.

\medskip

\noindent The compact analytic surface $E = \bigcup_{i=1}^rE_i$ with simple normal crossings is of \textsl{Type III}${}_1$ (Figure~\ref{fig2}) if,  possibly after relabeling, the following hold: 
\begin{enumerate}
\item\label{typeIII1-1} All of the $E_i$ are smooth rational. 
\item\label{typeIII1-2} $E_i\cap E_j =\emptyset$ unless $j =\pm i$. In particular, the dual complex of $E$ is a point (if $r=1$) or a line segment (if $r\geq 2$).
\item\label{typeIII1-3}  Assume that $r\geq 2$. Let $D_{i,i+1} =E_i\cap E_{i+1}$. Then $D_{i,i+1}$ is a smooth rational curve. For $i =1,r$, there exist    connected  curves $C_1$, $C_r$ on $E_1$, $E_r$ respectively, where $C_1$ and $C_r$ are chains of smooth rational curves meeting transversally, with $(C_1\cdot D_{12})_{E_1}  = (C_r\cdot D_{r-1,r})_{E_r} =2$, and we have $K_{E_1} = \scrO_{E_1}(-C_1-D_{12})$ and $K_{E_r} = \scrO_{E_1}(-C_r-D_{r-1,r})$. For $2\leq i\leq r-1$, there exist two   curves $C_i'$ and $C_i''$, both chains of smooth rational curves, supporting effective divisors of self-intersection $0$, and not containing $D_{i-1,i}$ or $D_{i,i+1}$,  with
$$\left(C_i'\cdot D_{i-1,i}\right)_{E_i} = \left(C_i''\cdot D_{i-1,i}\right)_{E_i} = \left(C_i'\cdot D_{i, i+1}\right)_{E_i}  = \left(C_i''\cdot D_{i, i+1}\right)_{E_i} =1,$$
such that 
$$K_{E_i} = \scrO_{E_i}\left(-D_{i-1,i}-C_i'-D_{i,i+1}-C_i''\right).$$ 
  Moreover, $C =C_1 + C_2'+C_2''+\cdots + C_{r-1}'+ C_{r-1}''+C_r$ is a Cartier divisor on $E$, where $C'= \sum_{i=2}^{r-1}C_i'$ and $C''= \sum_{i=2}^{r-1}C_i''$ are connected.    
\item\label{typeIII1-4} In all cases, $\scrO_E(E) = \omega_E = \scrO_E(-C)$, where if $r=1$, $C$ is a cycle of smooth rational curves or an irreducible nodal curve.
\end{enumerate} 
\begin{figure}[htb]
\includegraphics[scale=0.7]{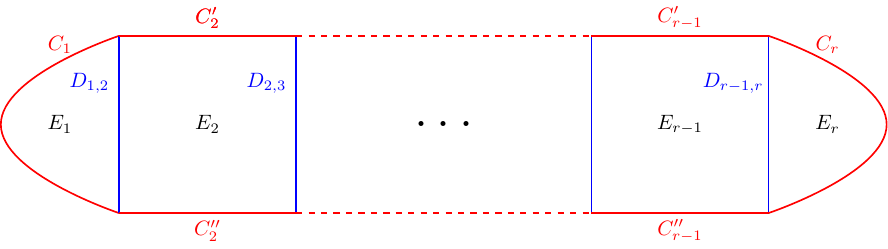}
\caption{Exceptional divisor $E$  of Type $III_{1}$}\label{fig2}
\end{figure}
Note that \eqref{typeIII1-4} implies that there exist points $p_i', p_i''$ in $D_{i, i+1}$ for $1\leq i\leq r-1$  which are smooth points on $D$ and such that $T^1_E\cong  \scrO_D(-\sum_i(p_i'+ p_i''))$.

\medskip

\noindent The compact analytic surface $E = \bigcup_{i=1}^rE_i$ with simple normal crossings is of \textsl{Type III}${}_2$ (Figure~\ref{fig3}) if $r\geq 2$ and, possibly after relabeling, the following hold: 
\begin{enumerate}
\item\label{typeIII2-1} All of the $E_i$ are smooth rational. 
\item\label{typeIII2-2} There are distinguished components  $E_1, \dots, E_s$ such that $E_i\cap E_j =\emptyset$ unless $j =\pm i \bmod s$. In particular, for $s\geq 3$, the dual complex of $E_1, \dots, E_s$ is a circle (and similarly for $r=2$, where the dual complex has two vertices joined by two edges). Moreover, there exist chains of smooth  rational curves  $C_i$ on $E_i$, $1\leq i \leq s$, such that  $C_i$  and $D_i$ have no component in common and  $C_i\cap E_j\neq \emptyset$ if and only if we have $1\leq j\leq s$ and $j =\pm 1\bmod s$. In this case,   $(C_i \cdot D_{i,i+1}')_{E_i} = 1$ for every component $D_{i,i+1}'$ of $D_{i,i+1}$ meeting $E_0$, with the convention that $D_{s,s+1}=D_{1,s}$.  Moreover, $C =C_1+\cdots + C_s$ is a Cartier divisor on $E$ contained in $\bigcup_{1=1}^sE_i$; more precisely,  $C_i\cap E_{i+1} = C_{i+1}\cap E_i$. 
\item\label{typeIII2-3}  For $1\leq i \leq s$, $C_i + D_i$ is a cycle of smooth rational curves,   and $K_{E_i} = \scrO_{E_i}(-C_i-D_i)$.  For $i> s$, $D_i$ is a cycle of smooth rational curves, and $K_{E_i} = \scrO_{E_i}(-D_i)$. 
\item\label{typeIII2-4} The dual complex of $E$ is  a semisimplicial triangulation of the $2$-disk, and $E_1 , \dots,  E_r$ are the boundary vertices.
\item\label{typeIII2-5} $\scrO_E(E) = \omega_E = \scrO_E(-\sum_{i=1}^sC_i)$. 
\end{enumerate} 
\begin{figure}[htb]
\includegraphics[scale=0.7]{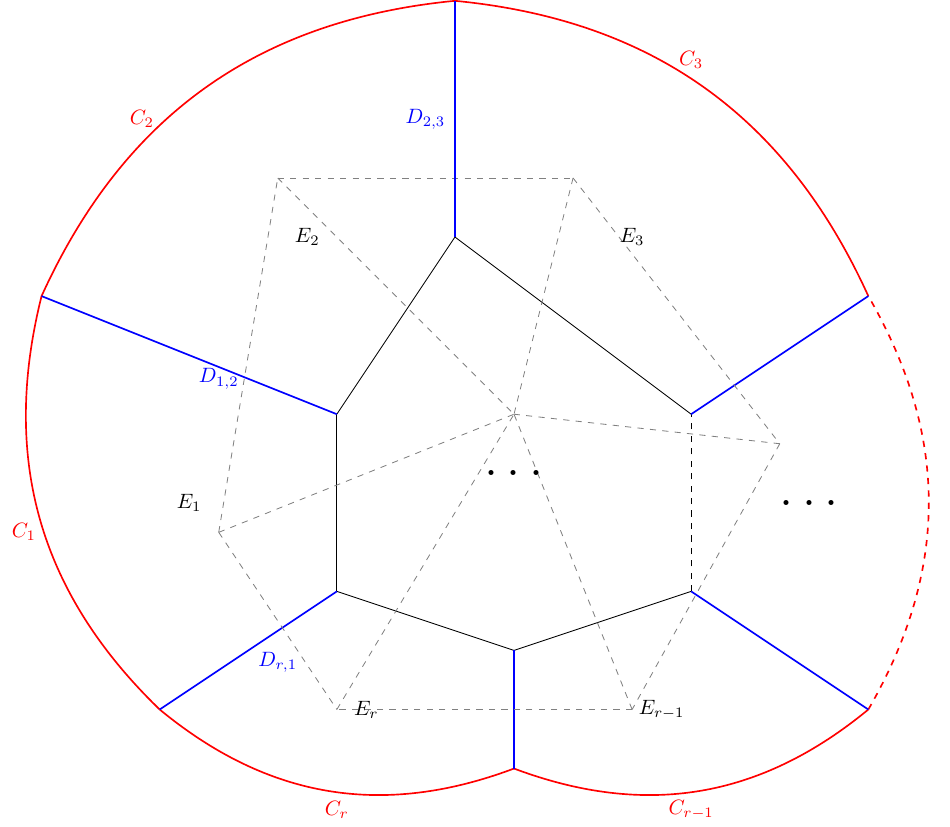}
\caption{Exceptional divisor $E$  of Type $III_{2}$}\label{fig3}
\end{figure}
Note that \eqref{typeIII2-5} implies the following (keeping the   convention that $D_{s,s+1}=D_{1,s}$): Suppose $D_{i,i+1}$ is irreducible. Then there exist points $p_i$ in $D_{i,i+1}$ for $1\leq i\leq s$, not in $D_{\text{\rm{sing}}}$, such that $T^1_E\cong \scrO_D(-\sum_i p_i)$. An analogous statement holds if $D_{i,i+1}$ is reducible, where such points exist in every component of $D_{i,i+1}$ meeting $C_i$. 
\end{definition}

\begin{remark}\label{Kulikov} The above terminology is modeled on the corresponding terminology for semistable degenerations of $K3$ surfaces due to Persson, Kulikov, and others. More precisely, consider a smoothing $\rho\colon X \to (\Delta,0)$, where $X = (X,x)$ is the germ of a simple elliptic or cusp singularity over the unit disk (\textit{i.e.} $\rho$ is a flat morphism of germs, $\rho^{-1}(0)$ is a simple elliptic or cusp singularity, and $\rho^{-1}(t)$ is smooth for $t\neq 0$). By \cite[Theorem 2.5]{FriedmanMiranda}, possibly after a base change, there exists a semistable crepant resolution $\pi\colon \hX \to X$. In other words, $\hX$ is smooth and induces  a minimal resolution of $\rho^{-1}(0)$ with exceptional divisor $D_0$, $K_{\hX}\cong \scrO_{\hX}$,  and the induced morphism $f= \rho\circ \pi\colon \hX\to \Delta$ satisfies the following: $f^{-1}(0)$ is a reduced divisor with simple normal crossings. Then the dual complex of $f^{-1}(0)$ is a line segment in the simple elliptic case (Type II) or a triangulation of $S^2$ in the cusp case (Type III). Thus, the dual complex of $E =\pi^{-1}(0)$ is again a line segment in the simple elliptic case, and hence is of Type II in the sense of Definition~\ref{defcrep1}, or it is obtained by deleting a vertex in a triangulation of $S^2$ and hence is either a line segment or a disk, \textit{i.e.} is of Type III${}_1$ or Type III${}_2$ in the sense of Definition~\ref{defcrep1}. See also Theorem~\ref{partialconverse} for a partial converse to this picture. 
\end{remark}

\begin{remark} As noted in greater generality in \cite{Kawamata1988}, the number $r$ of divisors in the exceptional set $E$ of the crepant resolution $\pi\colon \hX \to X$ is independent of the choice of crepant resolution. This follows easily from the fact  that any two crepant resolutions are isomorphic in codimension $1$ (see \textit{e.g.} Remark~\ref{crepremark}). The invariant $\ell= \dim H^2(L)$ defined in Proposition~\ref{3dimtop}\eqref{3dimtop-5} can also be computed directly in this case, at least for a semistable smoothing. 
\end{remark} 

\begin{proposition}\label{linkcalc}\leavevmode
  \begin{enumerate}
  \item[\rm(i)] Suppose that $\hX$ is the semistable model for a smoothing of a simple elliptic singularity, of multiplicity $m = -D_0^2$. Then
$$\ell = 9-m.$$
\item[\rm(ii)] Suppose that $\hX$ is the semistable model for a smoothing of a cusp singularity, of multiplicity $m = -D_0^2$ and length $s$, the number of components of the cusp. Then
$$\ell = 9-m+ s.$$
\end{enumerate}
\end{proposition} 
\begin{proof} By Proposition~\ref{3dimtop}\eqref{3dimtop-5}, $\ell =   \sum_{i=1}^rb_2(E_i) - \#\{\text{double curves}\} -r$, where $r$ is the number of components of $E$.  In the simple elliptic case, after standard birational operations (flops or Type II elementary modifications), we can assume that $E_1, \dots, E_{r-1}$ are minimal elliptic ruled surfaces and that $E_r$ is a generalized del Pezzo surface of degree $-(D_{r-1,r})_{E_{r-1}}^2 =(D_{r-1,r})_{E_r}^2 =  K_{E_r}^2 = m$. Thus there are $r-1$ double curves, $r-1$ elliptic ruled components $E_i$ with $b_2(E_i)=2r$, and the remaining component $E_r$ satisfies  $b_2(E_r)= 10-m$. Putting this together gives
$$\ell = 2(r-1) + 10-m - 2r +1 = 9-m.$$
In the cusp case,  we shall just write down the proof for $E$ of Type III${}_2$ (the proof in the Type III${}_1$ case is similar but simpler).  Let $e$ be the number of double curves of $E$, and let $t$ be the number of triple points. By taking Euler characteristics,
$$r-e+t=1.$$
Each surface $E_i$ satisfies $-K_{E_i} =\scrO_{E_i}(D_i+C_i)$ or $-K_{E_i} =\scrO_{E_i}(D_i)$, depending on whether $E_i$ meets $E_0$, and in this case $C_i$ is irreducible since by assumption  $\hX$ is  semistable. Set $\widehat{D}_i = D_i$ if  $E_i$ does not meet  $E_0$ (\textit{i.e.} $i > s$), and  set $\widehat{D}_i = D_i+C_i$ if  $E_i$  meets  $E_0$ (\textit{i.e.} $1\le i \le s$). If $s_i$ is the number of components of $\widehat{D}_i$, then $\sum_is_i = 2e+ s$. Every triple point is contained in  three edges, and every edge contains two triple points except for the edges corresponding to the double curves $E_i\cap E_{i+1}$, $1\le i \le s$, which just meet one triple point.   Thus $2e = 3t+s$. 

The surfaces $E_i$, $i > 0$, are rational surfaces with an anticanonical cycle $\widehat{D}_i $. Following \cite[Definition 3.1]{FriedmanMiranda},  define the \textsl{charge} $Q(E_i, \widehat{D}_i )$ by
$$  Q\left(E_i, \widehat{D}_i \right) = 12- \left(\widehat{D}_i\right)^2 - s_i.$$
Note that this definition makes sense for $E_0$ as well, where we set $\widehat{D}_0 =D_0 = \sum_{i=1}^rC_i$ and  $  Q(E_0, \widehat{D}_0) = 12+m - s$. For $i > 0$, by \cite[Lemma 1.2]{Fsurvey}, 
$b_2(E_i ) = Q(E_i, \widehat{D}_i ) -2 + s_i$. 
Then
{\allowdisplaybreaks
\begin{align*}
\ell &= \sum_{i\ge 1} b_2(E_i) - e -r  = \sum_{i\ge 1} Q\left(E_i, \widehat{D}_i \right) - 3r + \sum_{i\ge 1}s_i -e\\
&= \sum_{i\ge 1} Q\left(E_i, \widehat{D}_i \right) -3(1+e-t) + 2e+s -e \\
&= \sum_{i\ge 1} Q\left(E_i, \widehat{D}_i \right) -3 + (3t- 2e+ s) = \sum_{i\ge 1} Q\left(E_i, \widehat{D}_i \right) -3.
\end{align*}} 
The  principle of ``conservation of charge,'' see   \cite[Proposition 3.7]{FriedmanMiranda}, implies that 
$$\sum_{i\ge 1} Q\left(E_i, \widehat{D}_i \right) +  Q\left(E_0, \widehat{D}_0\right) = 24.$$
Thus $\sum_{i\ge 1} Q(E_i, \widehat{D}_i ) = 12 -m +s$, and so finally 
\begin{equation*}\pushQED{\qed}
  \ell = \sum_{i\ge 1} Q\left(E_i, \widehat{D}_i \right)  -3 = 9-m+s.
\qedhere \popQED
	\end{equation*}  
\renewcommand{\qedsymbol}{}
\end{proof}

\begin{remark}   Proposition~\ref{linkcalc} implies    the well-known results that  $m\leq 9$ for a smoothable simple elliptic singularity and $m \leq 9+s$ for a smoothable cusp singularity.
\end{remark}

In the Type II and Type III${}_1$ cases, $T^1_E$ is uniquely specified as noted in the definition. For the Type III${}_2$ case, $T^1_E$ is also uniquely specified by condition~\eqref{typeIII2-5}. 

\begin{lemma} Suppose that $E$ is of\, Type III${}_2$. For every $i$, $1\leq i\leq s$, and every component $\Gamma_{i\alpha}$ of\, $D_{i,i+1}$ meeting $C_i$,  choose points $q_{i\alpha}\in \Gamma_{i\alpha}$, not in $E_j$ for $j\neq i$. Then  $\scrO_D(-\sum_{i,\alpha}q_{i\alpha}) \cong T^1_E$. 
\end{lemma}
\begin{proof} Let $D =D'\cup D''$, where $D''$ consists of the components of $D_{i,i+1}$ meeting $C_i$, $1\leq i\leq s$, with the usual conventions, and $D'$ is the union of all of the other components. For simplicity, we just consider the case where $D_{i,i+1}\cong \Pee^1$ is irreducible for every $i$ and write $q_i$ instead of $q_{i\alpha}$. Let $L$ be a line bundle on $D$ such that $L|D'\cong \scrO_{D'}$ and $L|D''\cong \scrO_{D''}$. We claim that $L\cong \scrO_D$. Applying this to $\scrO_D(-\sum_iq_i) \otimes T^1_E \cong \scrO_D(p_i-\sum_iq_i)$ then proves the lemma.

More precisely, we shall show  that $\Pic D \cong \Pic D'\times \Pic D''\cong \Pic D'\times \Zee^s$,  and hence $p_a(D) = p_a(D')$.  We have an exact sequence
$$0\lra \scrO_D^* \lra \scrO_{D'}^* \times \scrO_{D''}^* \lra \prod_{i=1}^s\Cee^*_{p_i} \lra 0,$$
where $D'\cap D'' =\{p_1, \dots, p_s\}$ and $\Cee^*_{p_i}$ is the skyscraper sheaf at $p_i$ with stalk $\Cee^*$. 
But $H^0(D'';\scrO_{D''}^*)\cong (\Cee^*)^s$, and the induced map 
$$H^0\left(D'';\scrO_{D''}^*\right) \lra H^0\left( D;\prod_{i=1}^s\Cee^*_{p_i}\right) \cong (\Cee^*)^s$$
coming from the above exact sequence is an isomorphism. Thus
$$H^1\left(D;\scrO_D^*\right) =\Pic D \cong H^1\left(D';\scrO_{D'}^*\right) \times H^1\left(D'';\scrO_{D''}^*\right) = \Pic D'\times \Pic D'',$$ 
completing the proof.
\end{proof}

Next we  describe the cokernel of $H^1(\hX; T_{\hX})$ in $H^1(E;T^0_E)$. As we have seen in the proof of Theorem~\ref{goodcrepdef}\eqref{goodcrepdef-6}, this cokernel is $H^2(\hX; T_{\hX}(-E))$, and it measures the failure of a first-order locally trivial deformation of $E$ to be realized by a deformation of $\hX$ and the divisors $E_1, \dots, E_r$. 

\begin{proposition}\label{prop6} Suppose that $E$ is  of Type II. Then    $\dim H^2(\hX; T_{\hX}(-E)) \geq r-1$, where $r$ is the number of components of\, $E$, and $ H^2(\hX; T_{\hX}(-E)) =0$ if $r=1$, \textit{i.e.} $E$ is irreducible.
\end{proposition}

\begin{remark} The geometric interpretation of Proposition~\ref{prop6} is as follows. By Theorem~\ref{goodcrepdef}, $ H^2(\hX; T_{\hX}(-E))$ is the obstruction to realizing all first-order locally trivial deformations of $E$ by deforming $\hX$. We have seen that, in the Type II case, where $E \subseteq \hX$, we have $T^1_E \cong \scrO_D$, where $D=E_{\text{sing}}$. For a general deformation of $E$, the line bundle $T^1_E$ on $D$ has degree $0$ but is not necessarily the trivial line bundle. Here,  $r-1$ is the number of conditions on the deformation required to keep $T^1_E$ trivial. 
\end{remark}

\begin{proof}[Proof of Proposition \ref{prop6}]  By the formal functions theorem,
  $$H^2\left(\hX; T_{\hX}(-E)\right) =\varprojlim_n H^2\left(\hX; T_{\hX}(-E)|nE\right).$$
Since $H^3(\hX; T_{\hX}(-2E)) =0$,  $H^2(\hX; T_{\hX}(-E)) \to H^2(E; T_{\hX}(-E)|E )$ is surjective. For $n\geq 1$, we have  the exact sequence
$$0 \lra T_{\hX}(-E)|E \otimes \scrO_E(-nE) \lra T_{\hX}(-E)|(n+1)E \lra T_{\hX}(-E)|nE \lra 0. $$
By duality (where we also allow the cases $n=0, -1$), $H^2(E;T_{\hX}|E \otimes \scrO_E(-(n+1)E))$ is dual to 
$$H^0\left(E;\Omega^1_{\hX}|E \otimes \scrO_E((n+1)E) \otimes \omega_E\right) = H^0\left(E;\Omega^1_{\hX}|E \otimes \scrO_E((n+2)E)\right).$$
By Theorem~\ref{somecrepresults}\eqref{somecrepresults-1},  $\scrO_E(E) = K_E = \scrO_E(-D_0)$ for an effective nonzero divisor $D_0$ on $E$, disjoint from the singular locus, and similarly $\scrO_E(mE) =   \scrO_E(-mD_0)$. In particular,  for $m\geq 0$, there is an inclusion  $\scrO_E(mE) \to \scrO_E$ which only vanishes along the divisor $D_0$ for $m>0$. 
Now use the conormal sequence
$$0\lra I_E/I_E^2 \lra \Omega^1_{\hX}|E \lra \Omega^1_E \lra 0.$$
Since $I_E/I_E^2  =\scrO_E(-E)$, we have $I_E/I_E^2 \otimes \scrO_E((n+2)E)= \scrO_E((n+1)E)$. First assume $r=1$, \textit{i.e.} $E$ is smooth. From the conormal sequence, we get an exact sequence
$$H^0\left(E; \scrO_E((n+1)E)\right) \lra H^0\left(E;\Omega^1_{\hX}|E \otimes \scrO_E((n+2)E)\right) \lra H^0\left(E;\Omega^1_E \otimes \scrO_E((n+2)E)\right).$$
Since $H^0(E;\scrO_E((n+1)E))=0$ for all $n\geq 0$,   there is a sequence of  inclusions
$$H^0\left(E;\Omega^1_{\hX}|E \otimes \scrO_E((n+2)E)\right) \lra H^0\left(E;\Omega^1_E \otimes \scrO_E((n+2)E)\right) \lra H^0\left(E;\Omega^1_E\right) = 0.$$
Hence $H^0(E;\Omega^1_{\hX}|E \otimes \scrO_E((n+2)E)) =0$ for all $n\geq 0$, and therefore $H^2(\hX; T_{\hX}(-E)) =0$. 

Now assume $r\geq 2$.  With notation as in Definition~\ref{defcrep1}, $\scrO_E(mE)|D_{i,i+1}\cong \scrO_{D_{i,i+1}}$, $\scrO_E(mE)|E_i\cong \scrO_{E_i}$ if $i>1$, and $\scrO_E(mE)|E_1\cong \scrO_{E_1}(-mD_0)$. Using the exact sequence (from the normalization)
$$0\lra \scrO_E(mE) \lra \bigoplus_i\scrO_E(mE)|E_i \lra \bigoplus_i\scrO_{D_{i,i+1}} \lra 0,$$
we see that  the   map  $\bigoplus_iH^0(E_i; \scrO_E(mE)|E_i) \to \bigoplus_iH^0(D_{i,i+1};\scrO_{D_{i,i+1}})$ is an isomorphism.  Thus, as in the case $r=1$,  $H^0(E;\scrO_E(mE))=0$ for all $m>0$.

 As in the case $r=1$, we want to analyze $H^0(E;\Omega^1_{\hX}|E \otimes \scrO_E((n+2)E))$.  By the above,   
 $$H^0\left(E;\Omega^1_{\hX}|E \otimes \scrO_E((n+2)E)\right) \subseteq H^0\left(E;\Omega^1_E \otimes \scrO_E((n+2)E)\right).$$
 Now we have the exact sequence
$$0\lra \tau^1_E \lra  \Omega^1_E \lra \Omega^1_E/\tau^1_E \lra 0.$$ Also, 
$H^0(E; \Omega^1_E/\tau^1_E)=0$ by Proposition~\ref{3dimtop}\eqref{3dimtop-3}. Since there is an  inclusion 
$$H^0\left(E; \left(\Omega^1_E/\tau^1_E \right)\otimes \scrO_E((n+2)E)\right)  \subseteq H^0\left(E; \Omega^1_E/\tau^1_E\right),$$
we conclude that 
$H^0\left(E; \left(\Omega^1_E/\tau^1_E \right)\otimes \scrO_E((n+2)E)\right) =0$. The remaining term is  
$$H^0(E; \tau^1_E\otimes \scrO_E((n+2)E)).$$ As $D$ is smooth and the curve $D_{i,i+1}$ appears twice in $\coprod _iE_i$, by \cite[Proposition~1.10(2)]{F83}, there is an exact sequence
$$0\lra \scrO_{\hX}(-E)|D  \lra  \bigoplus_i\left(\scrO_{\hX}(-E)|D_{i,i+1}\right)^2 \lra \tau^1_E \lra 0.$$
Since $D$ is a disjoint union of the smooth components $D_i$ and $\scrO_{\hX}(-E)|D\cong \scrO_D$, this says that $\tau^1_E \cong \scrO_D = \bigoplus _{i=1}^{r-1} \scrO_{D_{i,i+1}}$. 
Hence
$$H^0\left(E; \tau^1_E\otimes \scrO_E((n+2)E)\right)\cong H^0(D; \scrO_D) =\Cee^{r-1}.$$
In particular, taking $n=1$, we get $\dim  H^2(E; T_{\hX}(-E)|E ) =r-1$. Thus $\dim H^2(\hX; T_{\hX}(-E)) \geq r-1$. 
\end{proof}

\begin{remark} If we wanted to fully calculate $\dim H^2(\hX; T_{\hX}(-E))$, we would have to understand the coboundary
$$H^1\left(nE; T_{\hX}(-E)|nE\right) \lra H^2\left(E; T_{\hX}(-E)|E \otimes \scrO_E(-nE)\right),$$
where the last part of the proof shows that 
$$\dim  H^2\left(E; T_{\hX}(-E)|E \otimes \scrO_E(-nE)\right) =r-1$$ for all $n\geq 1$. This seems difficult and  most likely depends on the higher infinitesimal neighborhoods of $E$ in $\hX$.
\end{remark}

 In the Type III${}_1$ case (and also the Type III${}_2$ case if every component of $D$ meets $C$), then in fact $H^2(\hX; T_{\hX}(-E)) =0$. 

\begin{proposition}\label{prop2.12TypeIII1} In the Type III${}_1$ case,   $ H^2(\hX; T_{\hX}(-E)) =0$. 
\end{proposition}
\begin{proof} The proof is similar to the proof of Proposition~\ref{prop6},  but in this case $\tau^1_E\otimes \scrO_E((n+2)E)$ is a line bundle of negative degree on every component of $D = E_{\text{sing}}$. Thus $H^0(E; \tau^1_E\otimes \scrO_E((n+2)E))=0$.
\end{proof}

\begin{remark}\label{remark2.12TypeIII2}  A similar (but also ultimately inconclusive) analysis along the lines of  Proposition~\ref{prop6} is possible in the Type III${}_2$ case:   One can show that, in this case, 
$$\dim H^0\left(E; \tau^1_E\otimes \scrO_E((n+2)E)\right)\geq p_a(D),$$
and in particular that  $\dim H^2(\hX; T_{\hX}(-E)) \geq p_a(D)$. It is easy to see that $ p_a(D)$ is  not $0$ in general. More precisely, using \cite[Lemma 2.25(iv)]{FriedmanEngel}, one can check that $ p_a(D) = r-s$, which is typically nonzero. 
\end{remark}

\section{The good crepant case: A partial classification}\label{Section3}
Let $(X,x)$ be an isolated   singularity  of dimension $3$ with a good crepant resolution $\pi \colon \hX \to X$.  It is natural to ask  if the (reduced) exceptional divisor  $E$ of $\pi$ is of Type II, Type III${}_1$, or Type III${}_2$. We have some partial results along these lines, inspired by the arguments of Shepherd-Barron \cite{SB1, SB2}. 

 \begin{theorem}\label{somecrepresults} With notation and assumptions as above, let $t$ be a general element of $ \mathfrak{m}_{X,x}$ in the sense of Definition~\ref{definegeneral}, and assume that the divisor $S=(t)$ defined by $t$ has a simple elliptic or cusp singularity.
 \begin{enumerate}
 \item\label{somecrepresults-1} With $(\pi^*t) = E_0+\sum_{i\geq1}a_iE_i$, we have $a_i = 1$ for all $i$. Hence
$$\scrO_{\hX}\left(\sum_{i\geq1}E_i\right) \cong \scrO_{\hX}(-E_0).$$
 \item\label{somecrepresults-2} If the   hypersurface section of\, $X$ defined by $t$  is simple elliptic, then $E$ is of Type II. In particular, all components of\, $E$ are rational or elliptic ruled.
 \item\label{somecrepresults-3}   If the  hypersurface section of\, $X$ defined by $t$  is a cusp, then all components of\, $E$ are rational and every component of a double curve  $D_{ij}$ is a smooth rational curve. $($The proof gives much more detailed information about this case.$)$
 \end{enumerate}
 \end{theorem}
\begin{proof} We begin with two lemmas.

\begin{lemma}\label{lemma3} Let $S$ be  the germ of a simple elliptic singularity or a cusp singularity, with singular point $x$, and let $p\colon \widetilde{S} \to S$ be a birational morphism such that $\widetilde{S}$ is normal, $p^{-1}(x)$ is a curve, \textit{i.e.} has dimension $1$, and $K_{\widetilde{S}} =\scrO_{\widetilde{S}}(-D_0)$, where $D_0$ is an effective nonzero Weil divisor whose support is contained in $p^{-1}(x)$. 
\begin{enumerate}
\item If\, $S$ is simple elliptic, then $D_0$ is a $($reduced$)$ smooth elliptic curve, $S$ is Gorenstein, and $D_0$ is a Cartier divisor  with $D_0^2< 0$.
\item If\, $S$ is a cusp, then $D_0$ is a reduced cycle of smooth rational curves or  an irreducible  nodal curve.
\end{enumerate}
\end{lemma}
\begin{proof} First consider the simple elliptic case. There exists a resolution of singularities $T\to S$ which dominates $\widetilde{S}$, and hence $T$ is the blowup of a minimal resolution of singularities of $S$, say $T_0$ with $K_{T_0} = \scrO_{T_0}(-C)$. Thus $T$ is obtained by successively blowing up points on $T_0$. All blowups at points not on the proper transform of $C$ correspond to components of $K_T$ occurring with positive coefficients: as Cartier divisors, $K_T =  -C' + \sum_ia_ie_i$, where the $e_i$ are proper transforms of exceptional curves corresponding to blowups at points not on the proper transform of $C$ and $a_i > 0$. Let $D_0'$ be the nonzero effective Cartier divisor on $T$ which is the proper transform of $D_0$. Then 
$$K_T  = -C'   + \sum_ia_ie_i =  -D_0'   + \text{ a sum of exceptional fibers of the morphism $T \to \widetilde{S}$}.$$
 The only way this is possible is if all  of the $e_i$ are fibers of the morphism $T \to \widetilde{S}$ and $C'$ is not an exceptional fiber of the morphism. It follows   that $\widetilde{S}$ is dominated by a surface $T$, a blowup of $T_0$  where all blowups are at points on the proper transform of $C$ and $C$ itself is not blown down. Hence $\widetilde{S}$ has at worst $A_n$ singularities, so is Gorenstein, and $K_{\widetilde{S}}= \omega_{\widetilde{S}} =\scrO_{\widetilde{S}}(-D_0)$, where $D_0$ is the image of the proper transform of $C$. Thus $D_0$ is a (reduced) smooth elliptic curve and $D_0^2< 0$.

A similar argument in the cusp case shows that if $T_0$ is the minimal resolution and $K_{T_0} = \scrO_{T_0}(-\sum_k\Gamma_k)$, where $\sum_k\Gamma_k$ is a cycle of rational curves  on $T_0$, then $\widetilde{S}$ is dominated by a surface $T$, a blowup of $T_0$  where all blowups are at points on the proper transform of $\sum_k\Gamma_k$ and $\sum_k\Gamma_k$ itself is not entirely blown down. In particular, the image $D_0$ of  $\sum_k\Gamma_k$ is a reduced cycle of smooth rational curves  or  an irreducible  nodal  curve.
\end{proof}

\begin{lemma}\label{connectedanticanon} Let $T$ be a smooth algebraic surface such that $-K_T =\sum_ia_iD_i$, where the $D_i$ are irreducible curves,   $a_i> 0$, and the sum is nonempty. Then either $T$ is an elliptic ruled surface and $-K_T = \sigma'+ \sigma''$, where $\sigma', \sigma''$ are disjoint smooth elliptic curves, or $\bigcup_iD_i$ is connected.
\end{lemma}
\begin{proof}   First, if $\rho\colon T \to \overline{T}$ is the  blowdown of an exceptional curve, so that $\overline{T}$ is smooth, then it is easy to check that $-K_{\overline{T}} = \sum_ia_i\rho_*(D_i)$ and that $T$ is a blowup of $\overline{T}$ at a point of $\bigcup_i\rho_*(D_i)$. Then $\bigcup_iD_i$ is connected if and only if $\bigcup_i\rho_*(D_i)$ is connected. Thus, we may as well assume that $T$ is  minimal, and hence is either $\Pee^2$, where the result is automatic, or the blowup of a ruled surface over a curve of genus $g$. In this case, using \cite{Hartshorne} as a general reference,  let $e$ be the invariant of the ruled surface, \textit{i.e.} $-e$ is the minimal self-intersection of a curve on $T$, and let $\sigma_0$ be a curve of self-intersection $-e$. In particular, $\sigma_0$ is a section of the ruling, and 
$$K_T \equiv -2\sigma_0 + (2g-2-e)f,$$
where $f$ is the numerical equivalence class of a fiber. Also note that either $e> 0$ and $\sigma_0$ is the unique curve on $T$ with negative self-intersection, or $e\leq0$ and every curve on $T$ has nonnegative self-intersection.

If $-K_T = D'+ D''$, where $D'$ and $D''$ are disjoint and nonempty effective divisors, then at most one of $D'$, $D''$ can have negative self-intersection. Thus we can assume that $(D'')^2\geq 0$. If  $(D'')^2 > 0$, then $(D')^2< 0$ by the Hodge index theorem; hence $\sigma_0$ is a component of $D'$. Then every component of $D''$ is disjoint from $\sigma_0$, hence is numerically equivalent to a positive multiple of $\sigma_0 + ef$. It follows that $D''$ is numerically equivalent to $\sigma_0 + ef$ and $D'=\sigma_0$. Moreover, $-K_T = D'+ D''$ is a union of two disjoint smooth sections. The argument also shows that if $(D'')^2= 0$, then $(D')^2= 0$ as well, and hence $D'$ and $D''$ are numerically equivalent. In particular, since $K_T \equiv -2\sigma_0 + (2g-2-e)f$, there exist disjoint sections $\sigma'\subseteq \operatorname{Supp} D'$ and $\sigma''\subseteq \operatorname{Supp} D''$, and all remaining components of $D'$, $D''$ are fibers. Since $D'$ and $D''$ are disjoint, we must have $\sigma' =D'$ and $\sigma'' = D''$. In all cases, if $\bigcup_iD_i$ is not connected, $-K_T = \sigma'+ \sigma''$, where $\sigma', \sigma''$ are disjoint sections of $T$. Then 
$$2g(\sigma') -2= 2g(\sigma'') -2= K_S\cdot \sigma' + (\sigma')^2 = -(\sigma')^2 + (\sigma')^2 =0.$$
Hence $\sigma', \sigma''$ are disjoint smooth elliptic curves, and $T$ is elliptic ruled.
\end{proof}

Returning to the proof of Theorem~\ref{somecrepresults},  let $S$ be a general hypersurface section through $x$.  

\subsubsection*{The case where $\boldsymbol{(S,x)}$ is simple elliptic}

 Let $E_0$ be the proper transform of $S$. First we claim that $E_0$ is normal. In any case,  $E_0$ is  Gorenstein, and $\omega_{E_0}= \scrO_{E_0}(-D_0)$, where $D_0 =  \sum_{i\geq 1}a_i( E_0\cap E_i)$. Let $\nu \colon \widetilde{E_0} \to E_0$ be the normalization (necessarily Cohen--Macaulay), and suppose that $\nu$ is not an isomorphism. Then $\omega_{\widetilde{E_0}} = \omega_{\widetilde{E_0}/E_0}\otimes \nu^*\omega_{E_0}$. Since $E_0$ is Gorenstein, it must fail to be normal in codimension $1$. Then $\omega_{\widetilde{E_0}/E_0}$ is an ideal sheaf (the conductor ideal sheaf) of $\scrO_{\widetilde{E_0}}$; hence  $\omega_{\widetilde{E_0}/E_0}$ is a rank $1$ reflexive sheaf and  $\omega_{\widetilde{E_0}/E_0}\cong \scrO_{\widetilde{E_0}}(-F)$ for some effective nonzero divisor $F$ on $\widetilde{E_0}$. Since the set of nonnormal points of $E_0$ is contained in $E_0\cap E$, every component of $F$ is also a component of $\nu^{-1}D_0$. In particular, some component of $-K_{\widetilde{E_0}}$ is nonreduced. But this contradicts Lemma~\ref{lemma3}.

Thus, $E_0$ is normal.  By Lemma~\ref{lemma3}, $E_0$ is Gorenstein and $\omega_{E_0}= \scrO_{E_0}(-D_0)$, where $D_0$ is smooth elliptic and $(D_0)^2_{E_0} < 0$. In particular, 
$$\scrO_{E_0}(E_0) = \scrO_{E_0}(-D_0) = \scrO_{\hX}\left(- \sum_{i\geq 1}a_iE_i\right)|E_0.$$
Then $E_0$ meets $\bigcup_{i\geq 1}E_i$ for a unique $i$, say $i=1$, and  $a_1 =1$. Thus 
$$K_{E_1} =  \scrO_{E_1}(E_1)= \scrO_{\hX}\left(-E_0- \sum_{i\geq 2}a_iE_i\right)|E_0.$$
Moreover, $(D_0)^2_{E_1}=-(D_0)^2_{E_0} > 0$.
By Lemma~\ref{connectedanticanon}, there are two possibilities: 
\begin{enumerate}
\item $D_0 + \sum_{i\geq 2}a_i(E_1\cap E_i)$ is connected. Then $E_1\cap E_i=\emptyset$ for $i> 1$, so that $r=1$, $E=E_1$, and $K_{E_1} =  \scrO_{E_1}(-D_0)$.
\item  $E_1$ is elliptic ruled and $K_{E_1} =  \scrO_{E_1}(-D_0-\Gamma)$ for some smooth elliptic curve $\Gamma$ disjoint from $D_0$ with $\Gamma^2  < 0$ by the Hodge index theorem.
\end{enumerate}
 In the first case,    $E_1$ is rational since $K_{E_1} = (D_0)^2_{E_1} > 0$. In the second case,
$E_1$ meets $\bigcup_{i\geq 2}E_i$ for a unique $i$, say $i=2$,   $a_2 =1$, and $\Gamma = D_{12}$ with $(D_{12})^2_{E_2}= -(D_{12})^2_{E_1}>0$. Then we can repeat this analysis. Eventually, this process must terminate with a rational $E_r$. Moreover, $K_{E_r} = \scrO_{E_r}(-D_{r-1,r})$, with $(D_{r-1,r}^2)_{E_r}>0$, so that $K_{E_r}$ is nef and big. Hence we have shown that $E$ satisfies \eqref{defcrep1-1}--\eqref{defcrep1-4} of the Type II case of Definition~\ref{defcrep1}. Condition \eqref{defcrep1-5} follows since $ \sum_{i\geq 0}E_i =(\pi^*t)$ is pulled back from $X$. Note that  $i\geq 1$, 
\begin{align*}
T^1_E|D_{i,i+1} &= N_{D_{i,i+1}/E_i}\otimes N_{D_{i,i+1}/E_{i+1}}= \scrO_{\hX}(E_i+E_{i+1})|D_{i,i+1} \\
&= \scrO_{\hX}\left(\sum_{i\geq 0}E_i\right)\Big|D_{i,i+1}=\scrO_{D_{i,i+1}},
\end{align*}
as remarked after the definition of Type II.

\subsubsection*{The case where $\boldsymbol{(S,x)}$ is a cusp}

 Let $E_0$ be the proper transform of $S$. By the same arguments as in the simple elliptic case, $E_0$ is normal and Gorenstein, and $\omega_{E_0}= \scrO_{E_0}(-\sum_{k=1}^n\Gamma_k)$, where either each $\Gamma_k$ is smooth rational and the dual graph of the $\Gamma_k$ is a cycle, or $n=1$ and $C = \Gamma_1$ is irreducible. Moreover, $E_0$ is dominated by a smooth surface $\widetilde{E}_0$ which is a blowup of a minimal resolution of the cusp, and we can further assume that no fibers of $\widetilde{E}_0\to E_0$ are exceptional curves, \textit{i.e.} that $\widetilde{E}_0\to E_0$ is a minimal resolution. Thus $E_0$ is obtained from $\widetilde{E}_0$ by contracting chains of curves of self-intersection $-2$.  We have
$$\scrO_{E_0}(E_0) = \scrO_{E_0}\left(-\sum_{k=1}^n\Gamma_k\right) = \scrO_{\hX}\left(- \sum_{i\geq 1}a_iE_i\right) \Big|E_0.$$
After reindexing, we can assume that the components   $E_i$  of $E$ meeting $E_0$  are $E_1, \dots, E_s$, and the above shows that  $a_i =1$ for such $i$. Define $E_i\cap E_0 = C_i$.    Then $E_i\cap E_0 = C_i$ is a union of some of the $\Gamma_k$, and either the connected components of $C_i$ are  chains of rational curves,  or $C_i$ is a cycle of rational curves or an irreducible   curve of arithmetic genus $1$. By Lemma~\ref{connectedanticanon}, the last two cases can only arise if $r=1$ and $E =E_1$ is irreducible, and necessarily a rational surface.  If  $r\geq 2$, then   $C_i$ is a disjoint union of chains of smooth rational curves, and each $E_i$ with $1\leq i\leq s$ meets at least one other component $E_j$ for which $C_i\cap C_j\neq \emptyset$. In fact, we can reorder the $E_i$ so that, for $1\leq i\leq s-1$, $E_i$ meets $E_{i+ 1}$ with $E_i\cap E_{i+1}\cap E_0 \neq \emptyset$. Note that, as $C_i$ is reduced,  no component of $C_i$ is contained in $E_j$ for $j\neq i$, and thus $D_{ij}$ is not contained in $C_i$  for $j\neq i$.

Taking for example $i=1$, $E_1$ meets $E_2$  at a point of $C_1\cap C_2$. Then  there exists a component $G_2$ of $D_{12}$ meeting the chain $C_1$, necessarily at an end component of the chain.  Now 
$$-K_{E_1} = C_1 + D_{12} +  \sum_{j\neq 1,2}a_jD_{1j},$$
and the support $C_1 \cup D_{12} \cup\bigcup_{j\neq 1,2}D_{1j}$ is connected by Lemma~\ref{connectedanticanon}. Hence
$$(K_{E_1}+G_2)\cdot  G_2 = -(C_1\cdot G_2) -\sum_{j\neq  1,2}a_j(D_{1j} \cdot G_2)< 0.$$ 
Thus $(K_{E_1}+G_2)\cdot G_2 =-2$, and $G_2$ is a smooth rational curve. Also, $(C_1\cdot G_2)$ is either $1$ or $2$. If it is~$2$, then $G_2 =  D_{12}$, $C_1$ is connected, and $E_1$ is a rational surface with $(D_{1j} \cdot  D_{12}) =0$ for $j\neq 1,2$. This implies that $E_1\cap E_j =\emptyset$ for $j\neq 1,2$. Also, by the connectedness of $C_1$, it follows that $E_2\cap E_j =\emptyset$ for $j\neq 1,2$, and we are in the case $s=2$. So we can assume that $(C_1\cdot G) =1$, there is a unique $k\neq 1,2$, say $k=3$, such that $(D_{13} \cdot G_2)\neq 0$, and  $(D_{13} \cdot G_2)=1$ and $a_3 =1$.

 Let $G_3$ be the unique component of $D_{13}$ meeting $G_2$. Repeating this argument with  $G_3$, we see  that $G_3$ is smooth rational and that  $(C_1+ D_{12})\cdot G_3$ is either $1$ or $2$. If it is $2$, then the only possibility is that $G_3$ meets $C_1$ at the other end of the chain from $G_2$. In this case, $E_1\cap E_j =\emptyset$ for $j\neq 1,2,3$. Otherwise, we can continue this process with an $E_j$ and with  a smooth component $G_j$ of $D_{1j}$ such that  $G_j\cdot G_3\neq 0$. Eventually the curves $C_1$ and the $G_j$  must close up (although it is possible for some $E_j$ to be equal to $E_\ell$ at an intermediate stage).  We can  do this analysis for all $E_i$, $1\leq i \leq s$:  Every component of $D_{ij}$ is a smooth rational curve, and $a_j =1$ for  every $j$ such that $E_i\cap E_j\neq \emptyset$.  Moreover, the scheme-theoretic intersection of $E_0$, $E_i$, and $E_{i\pm 1}$ is a reduced point and hence is a smooth point of $E_0$. In particular, $C = \bigcup_iC_i$ is a cycle of smooth rational curves, hence has arithmetic genus $1$.

   Now let $E_k$ be a component of $E$ with $k> s$ such that $E_k\cap E_1\neq \emptyset$, say, and let $G$ be a component of $D_{1k}$. Then $a_k =1$, and $G$ is a smooth rational curve. Moreover, there exists an $\ell\neq i, j$ such that $G\cap E_\ell \neq \emptyset$ as well.  Thus  $E_1\cap E_k\cap E_\ell \neq \emptyset$, say, and
$$-K_{E_k} = D_{1k} + D_{k\ell} + \sum_{t\neq  1,k, \ell }a_tD_{kt}.$$
Let $G'$ be a component  of $D_{k\ell}$ meeting $G$. Arguments as above show that $G$ is smooth rational and that $a_t=1$ for every $t$ such that $E_t\cap G'\neq \emptyset$.  Continuing in  this way and using the connectedness of $E$, it follows that $a_i=1$ for every $i$ and that  every component of $D_{ij}$ is a smooth rational curve. Moreover, for every $i$, there exists a curve in  $|-K_{E_i}|$ whose components are rational curves. Hence $E_i$ is rational. Thus, $E$ satisfies the conditions of the second statement in Theorem~\ref{somecrepresults}\eqref{somecrepresults-3}. 
\end{proof}

\begin{remark} The proof shows that in case the hypersurface section $S$ is a cusp and using the notation of the proof,  $-K_{E_i}$ is effective and is a cycle  of smooth rational curves which contains $C_i$ if $E_i\cap E_0\neq \emptyset$.
\end{remark}

 \begin{remark}\label{somecrepresults2}\leavevmode
\begin{enumerate}[wide]
 \item It seems quite possible that, in general, $E$ might not be of Type III${}_1$ or Type III${}_2$, even after making some flops. In particular, from the point of view of the classification of algebraic varieties, it is reasonable to allow  $E$ to have more complicated singularities than normal crossings, namely dlt singularities.  
 \item  
  In the Type II case,    $\dim H^0(E; T^1_E) = r-1$ is the number of elliptic ruled components. In the Type III${}_1$ or Type III${}_2$ cases, $H^0(E; T^1_E) =0$, so that all first-order deformations of $E$ are locally trivial in these cases. This follows more generally in case the general hypersurface section of $X$ passing through $x$ is a cusp, by Theorem~\ref{somecrepresults}\eqref{somecrepresults-3} and Theorem~\ref{goodcrepdef}\eqref{goodcrepdef-2}. 
 \item   It is easy to see that if $|\Gamma|$ is the dual complex of $E$, then $H^i(|\Gamma|) =0$ for $i> 0$, and indeed a theorem of \cite{dFKX} says that $|\Gamma|$ is contractible. In the Type II and Type III${}_1$ cases, $|\Gamma|$ is a point or a line segment, and in the Type III${}_2$ case  $|\Gamma|$ is a disk.  However, in the general case, without making   flops, the topological type of the dual complex can be more complicated than a line segment or a disk. For example, it can be  a disk meeting a line segment at a point.
  \end{enumerate}
 \end{remark}
  
On the positive side, there is the following. 

\begin{theorem}\label{partialconverse} Let $(X,x)$ be an isolated   singularity  of dimension $3$ with a good crepant resolution $\pi \colon \hX \to X$, and let $E =\pi^{-1}(x)$ be the reduced exceptional divisor. 
  \begin{enumerate}
  \item\label{partialconverse-1} If the general hypersurface section of\, $X$ passing through $x$ is a cusp and $\omega_E^{-1}$ is nef and big, then $E$ is of Type III${}_1$ or Type III${}_2$.
 \item\label{partialconverse-2} If the general hypersurface section $S$ of\, $X$ passing through $x$ is a cusp and the full inverse image  $\pi^{-1}(S)$ has normal crossings, then after a sequence of flops $($elementary modifications of type 2\,$)$, $\omega_E^{-1}$ becomes nef and big, hence $E$ is of Type III${}_1$ or Type III${}_2$.  
 \end{enumerate}
 \end{theorem}
 \begin{proof}  First assume that $\omega_E^{-1}$ is nef and big. In the notation of the proof of Theorem~\ref{somecrepresults}, $C = \bigcup_{i=1}^sC_i=E \cap E_0$ is a cycle of rational curves,   and  $\scrO_{E_0}(-C) = \omega_{E_0}=\scrO_{E_0}(-\sum_i\Gamma_i)$. Likewise, $\omega_E = \scrO_E(-C)= \scrO_E(E)$. Since $(C)^2_E = -(C )^2_{E_0} > 0$, the total degree of $\scrO_C(C)=\omega_E^{-1}|C$ is positive. Then general results on line bundles on cycles of rational curves (\textit{cf.} \cite[Lemma 1.7]{Fsurvey}) imply  that either $(C)^2_E \ge  2$ and $\scrO_C(C)$ has no base points, or $(C)^2_E =1$ and $\scrO_C(C)$ has a single base point at a smooth point of $C$. From the exact sequence
 $$0 \lra \scrO_E \lra \scrO_E(C) \lra \scrO_C(C) \lra 0,$$
 there exists a section of $\scrO_E(C)$ vanishing at   $C^* $, where $C^* = \bigcup_{i=1}^sC_i^*$ and $C_i^*$ is smooth for all $i$. 
 Since $\scrO_E(-E)|C$ is nef and has nonnegative degree on every component and positive total degree, $H^1(E;\scrO_E(-NE)) =0$ for all $N \ge 0$.  By induction,   it is then easy to see that $H^1(nE;\scrO_{\hX}(-NE)|nE)= 0$ for all $N \ge 0$ and all $n > 0$. Thus $R^1\pi_*\scrO_{\hX}(-NE) =0$ for all $N\ge 0$ by the formal functions theorem.  In particular, $R^1\pi_*\scrO_{\hX}(-2E) =0$. By applying $R^i\pi_*$ to the exact sequence
 $$0 \lra \scrO_{\hX}(-2E) \lra \scrO_{\hX}(-E) \lra \scrO_E(C) \lra 0,$$
 it follows that the natural map 
 $$R^0\pi_*\scrO_{\hX}(-E) \lra H^0(E; \scrO_E(C))$$
 is surjective. Hence, there exists an element $t\in R^0\pi_*\scrO_{\hX}(-E) \subseteq \mathfrak{m}_x$ which lifts to a function $\pi^*t$ whose restriction to $E$ is $C^*$. Thus $S^* = \{t=0\}$ defines a cusp singularity on $X$,  the proper transform $E_0^*$ of $S^*$ is a resolution of singularities of $S^*$, and $E\cup E_0^*$ has simple normal crossings. It follows from the classification of Type III degenerations of $K3$ surfaces that $E\cup E_0^*$ is a Type III anticanonical pair, i.e.\ which meets the description of  \cite[Lemma~2.14]{FriedmanMiranda}
or   \cite[Definition~2.1]{FriedmanEngel}, except that the Hirzebruch--Inoue component has been replaced by the local surface $E_0^*$.  The assumption that $\omega_E^{-1}$ is nef and big is then equivalent to the assumption that $(C_i')^2_{E_i} \ge 0$ for every component $C_i'$ of $C_i^*=E_0^*\cap E_i$.
 
 On the other hand,  under the assumptions of \eqref{partialconverse-2}, $E\cup E_0$ is again a  Type III anticanonical pair as above. Then every component $C_i'$ of $C_i=E_0\cap E_i$ is a smooth rational curve, and $(C_i')^2_{E_0} < 0$ for every~$i$. By the triple point formula, $(C_i')^2_{E_i} \ge 0$ unless $(C_i')^2_{E_0} = (C_i')^2_{E_i} = -1$. In this case, the standard flop (type 2 modification) eliminates $C_i'$ but does not alter the assumption that $\pi^{-1}(S)$ has normal crossings. In this process, the total number $i$ such that  $(C_i')^2_{E_0} =-1$ decreases, so it must ultimately terminate at a stage where $(C_i')^2_{E_0} \le -2$ for every $i$. Thus we can assume in \eqref{partialconverse-1}  and in \eqref{partialconverse-2} that $(C_i')^2_{E_i} \ge 0$ for every component $C_i'$ of $C_i=E_0\cap E_i$.
 
 In this case, suppose that there exists a component $E_1$ such that $C_1 = E_0 \cap E_1$ is disconnected. Then every component of $C_1$ has nonnegative square on $E_1$. By the Hodge index theorem, every component  of $C_1$ has square $0$ on $E_1$, and $E_1$ is a minimal rational ruled surface with the remaining double curves sections of the ruling. If $E_2$ is a component meeting $E_0$ and $E_1$, then the same analysis shows that either every component  of $C_2$ has square $0$ on $E_2$ and $E_2$ is a minimal rational ruled surface with the remaining double curves sections of the ruling, or $C_2$ is connected, hence an irreducible smooth rational curve, and $C_1$ has just two components. Continuing in this way, we see that  $E$ is of   Type III${}_1$.  In the remaining case, $E_0$ meets every component of $E$ in an irreducible smooth rational curve.
It is easy to see in this case that the dual complex of $E$ triangulates a $2$-disk and thus that $E$ is of   Type~III${}_2$.
 \end{proof}

\section{The case of a small resolution}\label{Section4} 

In this section, we  consider the case of a small resolution $\pi'\colon X' \to X$; \textit{i.e.}  $(X,x)$ is the germ of an isolated Gorenstein singularity of dimension $3$ with a good Stein representative $X$, and   $p\colon X'\to X$ is a small resolution with exceptional set  $C$.  

\begin{remark}\label{smallres} There is no real limitation to restricting to dimension $3$, at least in case the singularity $(X,x)$ is a local complete intersection. Indeed,  such resolutions can only exist for $\dim X =3$:  By the Grothendieck--Lefschetz theorem,  the local ring $\scrO_{X,x}$ of an isolated  local complete intersection singularity is a UFD for $\dim X \ge 4$.  The proof of the ``easy case'' of Zariski's Main Theorem shows that in case $\scrO_{X,x}$ is a UFD and $\pi\colon \hX \to X$ is a resolution of singularities, there exists a divisor $D$ on $\hX$ such that $\operatorname{codim} \overline{\pi(D)} \ge 2$. Thus, if $(X,x)$ is an isolated  local complete intersection singularity, and $\pi'\colon X' \to X$ is a resolution such that the exceptional set $(\pi')^{-1}(x)$ has dimension $\le \dim X - 2$, then $\dim X = 3$. 

Conversely, in dimension $3$, suppose that  $(X,x)$ is the germ of an isolated singularity with a small resolution $\pi'\colon X' \to X$. If $X$ is Cohen--Macaulay, then it follows from results of Laufer, Reid, and Pinkham (see for example \cite[Section~8]{PinkhamSurvey}) that $X$ is  a compound du Val singularity, \textit{i.e.} that the general hyperplane section of $X$ in the sense of Definition~\ref{definegeneral} is a rational double point. In particular, $X$ is a hypersurface singularity and $(\pi')^{-1}(x) =    C =\bigcup_{i=1}^rC_i$, where the $C_i$ are smooth rational curves meeting (pairwise) transversally (but three $C_i$ can meet at a point).
\end{remark}

For the case of a small resolution, the functor $\mathbf{Def}_{X'}$ has a more than purely formal meaning: By \cite[Theorem 2]{Laufer}, there is a deformation of a neighborhood of the exceptional curve $C$ over the smooth germ $(H^1(X';T_{X'}), 0)$ for which the Kodaira--Spencer map is an isomorphism.   As previously noted, $X'$ is a crepant resolution of $X$, and hence $K_{X'} \cong \scrO_{X'}$. Also, the resolution $p \colon X' \to X$ is  equivariant (\textit{cf.} Definition~\ref{defequi}).    
 
 We will use the following standard fact about local cohomology. 
 
 \begin{lemma}\label{localcohomolemma} If $\mathcal{F}$ is a locally free sheaf on $X'$, then $H^1_C(X';\mathcal{F})=0$.  
 \end{lemma}
 \begin{proof} Using the Mayer--Vietoris sequence, see \cite[Exercise~III.2.4, p.\ 212]{Hartshorne}, it suffices to show that $H^1_{C_i}(X';\mathcal{F})=0$ for every irreducible component $C_i$ of $C$ and that $H^2_p(X';\mathcal{F})=0$ for every point $p\in X'$. By \cite[Proposition~1.4]{HartshorneLC}, there is a spectral sequence with $E_1$ term $E_1^{p,q} = H^p(X'; \mathcal{H}_{C_i}^q(\mathcal{F}))$ converging to $H^{p+q}_{C_i}(X';\mathcal{F})$, where $\mathcal{H}_{C_i}^q(\mathcal{F})$ is the associated local cohomology sheaf. Since $C_i$ is smooth, it is a local complete intersection. Hence, by \cite[Proposition~3.7 and Theorem~3.8]{HartshorneLC},  $\mathcal{H}_{C_i}^q(\mathcal{F}) =0$ for $i=0,1$. Thus $H^1_{C_i}(X';\mathcal{F})=0$. The vanishing of $H^2_p(X';\mathcal{F})$ is similar. 
 \end{proof}

 Since the fibers of $p$ have dimension $1$,  $R^2p_*T_{X'} =0$ and hence $H^0_x(X;R^2p_*T_{X'})=0$. Applying the Leray spectral in local cohomology to the morphism $p$ and the sheaf $T_{X'}$ and using Lemma~\ref{localcohomolemma} to see that $H^1_C(X';T_{X'}) = 0$ gives the following.  

\begin{lemma}\label{newlemma41} There is an exact sequence
$$0 \lra  H^0\left(X;R^1p_*T_{X'}\right) \lra H^0\left(X;T^1_X\right) \lra H^2_C\left(X';T_{X'}\right) \lra  0. $$
\end{lemma}

Note that $H^0_x(X;R^1p_*T_{X'}) = H^0(X; R^1p_*T_{X'})$ and that $H^2_x(X;T^0_X) \cong H^0(X;T^1_X)$. 
Since $K_{X'}$ is trivial,   $T_{X'} \cong \Omega^2_{X'}$ and hence $H^2_C(X';T_{X'}) \cong H^2_C(X';\Omega^2_{X'})$.

The (singular) local cohomology groups $H^k_C(X')$ can be described via duality:
$$H^k_C(X') \cong H_{6-k}(C) =\begin{cases} 0 &\text{if $k\neq 4,6$,}\\
\bigoplus_iH_2(C_i) &\text{if $k= 4$,} \\
\Cee &\text{if $k=6$.}
\end{cases}$$

Moreover, there is  a spectral sequence 
$$E_1^{p,q}= H^q_C\left(X';\Omega^p_{X'}\right) \implies \mathbb{H}^{p+q}_C\left(X';\Omega^\bullet_{X'}\right) = H^{p+q}_C(X').$$
Many of the terms in the $E_1$ page of the spectral sequence are zero.

\begin{lemma} If $q=0,1$, then $H^q_C(X';\Omega^p_{X'}) =0$ for all $p$ and $H^2_C(X';\Omega^3_{X'})=  H^2_C(X';\scrO_{X'})= 0$.
\end{lemma}
\begin{proof} The first statement follows from Lemma~\ref{localcohomolemma}. The second follows by considering  the Leray spectral sequence with $E_2^{p,q}= H^p_x(X;R^qp_*\scrO_{X'}) \Rightarrow H^{p+q}_C(X';\scrO_{X'})$. Here, $R^qp_*\scrO_{X'} =0$ for $q>0$ since $(X,x)$ is a rational singularity and 
$$H^p_x\left(X;R^0p_*\scrO_{X'}\right) =  H^p_x\left(X;\scrO_X\right) = 0$$
for $p< 3$ because depth $\scrO_{X,x} = 3$.
\end{proof}   

Thus we have the following picture for the $E_1^{p,q}$ page of the spectral sequence converging to $\mathbb{H}_C^*(X'; \Omega^\bullet_{X'})= H^{p+q}_C(X')$:
\begin{center}\def\arraystretch{1.2}%
\begin{tabular}{|c|c|c|c}
$H^3_C\left(X';\scrO_{X'}\right)$ & $H^3_C\left(X';\Omega^1_{X'}\right)$ &$H^3_C\left(X';\Omega^2_{X'}\right)$& $H^3_C\left(X';\Omega^3_{X'}\right)$\\ \hline
{} & $H^2_C\left(X';\Omega^1_{X'}\right)$ & $H^2_C\left(X';\Omega^2_{X'}\right)$ &{} \\ \hline
{}&   &  &{}\\ \hline
{} &{}&{} &{}\\ \hline
\end{tabular}

\end{center}

\begin{lemma}\label{disinj}  The differential $d \colon H^2_C(X';\Omega^1_{X'}) \to H^2_C(X';\Omega^2_{X'})$ is injective.
\end{lemma}
\begin{proof}  This is clear since the kernel of $d \colon H^2_C(X';\Omega^1_{X'}) \to H^2_C(X';\Omega^2_{X'})$ would inject into \mbox{$H^3_C(X')=0$}.
\end{proof}

By examining the above spectral sequence, we see that there is a homomorphism $H^2_C(X';\Omega^2_{X'})\to H^4_C(X')$. 

\begin{proposition}\label{splitKandSsmall} The map $H^2_C(X';\Omega^2_{X'})\to H^4_C(X')$ is surjective and split by the fundamental class map. Thus, if $K_x'$ denotes the kernel of\, $H^2_C(X';\Omega^2_{X'})\to H^4_C(X')$, we have a direct sum decomposition
$$H^2_C\left(X';\Omega^2_{X'}\right) \cong K_x'\oplus  H^4_C(X')\cong K_x'\oplus \left(\bigoplus_{i=1}^r \Cee[C_i]\right).$$
\end{proposition} 
\begin{proof} First note that the fundamental classes of the $C_i$ are a basis for $H^4_C(X') \cong H_2(C)$. On the other hand,  for every $i$ one can construct a fundamental class $[C_i]\in H^2_{C_i}(X'; \Omega^2_{X'})$ which maps to the fundamental class $[C_i] \in H^4_C(X')$.  We also have the Mayer--Vietoris sequence
$$0= \bigoplus _{z\in C_{\text{\rm{sing}}}}H^2_z\left(X';\Omega^2_{X'}\right) \lra \bigoplus_iH^2_{C_i}\left(X';\Omega^2_{X'}\right) \to H^2_C\left(X';\Omega^2_{X'}\right),$$
and hence we can view the $[C_i]$ as linearly independent elements of $H^2_C(X';\Omega^2_{X'})$. The image of $H^2_C(X'; \Omega^2_{X'})$ in $H^4_C(X')$ therefore contains the vector space spanned by the fundamental classes of the components of $C$, and hence is equal to $H^4_C(X')$. Thus $H^2_C(X'; \Omega^2_{X'}) \to H^4_C(X')$ is surjective, and the subspace spanned by the fundamental classes of the components of $C$ is a complement to the kernel. 
\end{proof} 

Define $S_x\cong \Cee^r\subseteq H^2_C(\Omega^2_{X'})$ to be the image of the fundamental class  map. Thus $H^2_C(X';\Omega^2_{X'}) \cong K_x'\oplus S_x$. By Proposition~\ref{splitKandSsmall}, we can identify $S_x$ with $H^4_C(X')$. It can also be identified with $H^3(L)$, where $L = X -\{x\} = X'-C$ is the link of the singularity. This follows from the exact sequence
$$0 =  H^3(X') \lra H^3(U) \lra H^4_C(X') \lra H^4(X') = 0.$$
In particular, the mixed Hodge structure on $S_x \cong H^3(L)$ is of pure weight $4$ and of type $(2,2)$.

To say more about $K_x'$, we have the following. 

\begin{lemma}\label{smallsequence} Let $A_x$ be the kernel of $d\colon H^3_C(X';\scrO_{X'}) \to H^3_C(X';\Omega^1_{X'})$, and let $a = \dim A_x$. Then 
there is an exact sequence
$$0  \lra H^2_C\left(X';\Omega^1_{X'}\right) \lra K_x'   \lra A_x \lra 0.$$
Thus $\dim K_x' = \dim H^2_C(X';\Omega^1_{X'}) + a$. 
If $A_x=0$, then $d\colon H^2_C(X';\Omega^1_{X'}) \to K_x'$ is an isomorphism.
\end{lemma}
\begin{proof} By examining the spectral sequence, we see that $d_2\colon A_x \to H^2_C(X';\Omega^2_{X'})/d(H^2_C(X';\Omega^1_{X'}))$ is injective and that its image is $K_x'/d(H^2_C(X';\Omega^1_{X'}))$. Hence $A_x \cong K_x'/d(H^2_C(X';\Omega^1_{X'}))$. The remaining statements of the lemma are  clear.
\end{proof}

\begin{corollary}\label{smallfiltration} In the above notation,   there is an exact sequence
$$0 \lra (K_x')\spcheck/A_x\spcheck \lra H^0\left(X;T^1_X\right) \lra K_x' \oplus  S_x \lra 0.$$
\end{corollary}
\begin{proof}
By duality,  arguing as in \cite{karras}, 
$$H^0\left(X;R^1p_*T_{X'}\right) \cong H^2_C\left(X';\Omega^1_{X'}\right)\spcheck \cong (K_x')\spcheck/A_x\spcheck.$$
The proof then follows from Proposition~\ref{splitKandSsmall} and Lemma~\ref{smallsequence}.
\end{proof}

\begin{corollary}\label{smallnumerics} Let $b=\dim K_x'$. Then
\begin{enumerate}
\item\label{sn-1} $\dim H^2_C(X';\Omega^2_{X'}) = b+r$; 
\item\label{sn-2} $\dim H^2_C(X';\Omega^1_{X'}) = \dim H^0(X; R^1p_*T_{X'}) =  b-a$; 
\item\label{sn-3} $\dim H^0(X;T^1_X) = 2b-a+r$; thus $b+r \leq \dim H^0(X;T^1_X) \leq 2b+r$.
\end{enumerate}
\end{corollary}
\begin{proof} These follow from Proposition~\ref{splitKandSsmall}, Lemma~\ref{smallsequence}, and Corollary~\ref{smallfiltration} (and its proof). 
\end{proof}

\begin{theorem}\label{charodp} In the above situation, the following are equivalent:
\begin{enumerate}
\item\label{charodp-1} $C$ is smooth, \textit{i.e.} $r=1$, and the normal bundle satisfies  $N_{C/X'} \cong \scrO_{\Pee^1}(-1) \oplus  \scrO_{\Pee^1}(-1)$. In other words, $(X,x)$ is the germ of an ordinary double point.
\item\label{charodp-2} $\dim H^2_C(X';\Omega^2_{X'}) = r$, \textit{i.e.} $K_x' =0$, or equivalently $b=0$. 
\item\label{charodp-3} $R^1p_*T_{X'} = 0$, or equivalently $b=a$.
\end{enumerate}
\end{theorem} 
\begin{proof} \eqref{charodp-1} $\Rightarrow$ \eqref{charodp-2}~ In this case, $\dim H^0(X;T^1_X)=1 = 2b-a + 1$. Since $b\geq 0$ and $b-a =\dim H^2_C(X';\Omega^1_{X'})\geq 0$, we must have $b=0$, and hence $K_x' =0$.

\smallskip
\noindent \eqref{charodp-2} $\Rightarrow$ \eqref{charodp-3}~ By  Corollary~\ref{smallnumerics}\eqref{sn-2}, $\dim H^0(X;R^1p_*T_{X'}) =   b-a$. If $K_x' =0$, then $b=\dim K_x'=0$,   hence $a=0$ and $R^1p_*T_{X'} = 0$.

\smallskip
\noindent \eqref{charodp-3} $\Rightarrow$ \eqref{charodp-1}~ Following the discussion in \cite[Section~2, pp.\ 678--679]{F}, there exists a small deformation of $X'$ to $X'_t$ where the exceptional curve $C$ splits up into a union of $\delta$ disjoint copies of $\Pee^1$ with normal bundle $\scrO_{\Pee^1}(-1) \oplus  \scrO_{\Pee^1}(-1)$, and such a deformation blows down to a deformation of $X$ to a union of $\delta$ ordinary double points. But if $R^1p_*T_{X'} = 0$, then the only deformations of $X'$  
are locally trivial. Hence  the exceptional curve $C$ on $X'$ is already a single $\Pee^1$ with normal bundle $\scrO_{\Pee^1}(-1) \oplus  \scrO_{\Pee^1}(-1)$, and  $(X,x)$ is  the germ of an ordinary double point.
\end{proof}

\begin{proposition}\label{deltaconj}   Suppose that  there is a small deformation of $X'$ to a space   with exactly $\delta$ compact curves and that all of these have normal bundle   $\scrO_{\Pee^1}(-1) \oplus \scrO_{\Pee^1}(-1)$. Then $\dim H^2_C(X';\Omega^2_{X'}) =\delta$. 
 \end{proposition}
 \begin{proof} Let $(S,s_0)$ be the germ of a smooth analytic space prorepresenting $\mathbf{Def}_X$, and let $(S',s_0')$ prorepresent $\mathbf{Def}_{X'}$. The morphism of functors $\mathbf{Def}_{X'} \to \mathbf{Def}_X$ then induces a morphism of germs $S'\to S$ which is an immersion by \cite[Proposition 2.1]{F}.  By Lemma~\ref{newlemma41}, $H^2_C(X';\Omega^2_{X'})$ is the normal bundle to this immersion, so it will suffice to prove that the image of $S'$ has codimension $\delta$. As noted in the proof of Theorem~\ref{charodp}, there is an open dense subset of $S'$ corresponding to a germ with exactly $\delta$ singularities, all of which are ordinary double points. By the openness of versality, the image of $S'$ in $S$ then has codimension $\delta$ as claimed. 
 \end{proof}

\begin{corollary}[\textit{cf.} \protect{\cite[Lemma\ 1.9]{namstrata}}]
 We have $b+r =\delta$. Thus $\delta \geq r$, with equality if and only if      $(X,x)$ is an ordinary double point. \qed
\end{corollary}

We now compare the above discussion with the case of a good resolution, as described in \cite[Theorem 2.1]{FL}. By successively blowing up the curves $C_i$ in some order, we obtain a good resolution $\pi\colon \hX \to X$ which is an iterated blowup of $X'$. Let $\rho\colon \hX \to X'$ be the blowup morphism and $E=\bigcup_iE_i$ the exceptional divisor of $\rho$ or of $\pi$. To distinguish groups on $X'$ and on $\hX$, we denote the latter with a ``\, $\widehat{}$ \,.''  Then we have the group $\widehat{K}_x'$ defined in \cite[Theorem 2.1(vi)]{FL}: 
$$\widehat{K}_x' =\Ker\left\{H^2_E\left(\hX; \Omega^2_{\hX}\right) \lra H^2\left(\hX; \Omega^2_{\hX}\right)\right\}.$$
We also have $\Gr_F^2H^3(L) = H^3(L)$ and $\widehat{A}_x = \Ker\{d\colon H^3(\hX;\scrO_{\hX}) \to H^3_E(\hX; \Omega^1_{\hX})\}$. (In \cite[Corollary 1.8]{FL}, $\widehat{A}_x$ is defined to be the kernel of $d\colon H^3(\hX;\scrO_{\hX}) \to H^3_E(\hX; \Omega^1_{\hX}(\log E))\}$, but it is easy to check that in our case $H^3_E(\hX; \Omega^1_{\hX}) \to H^3_E(\hX; \Omega^1_{\hX}(\log E))$ is injective.) By \cite[Theorem 2.1(v), (vi)]{FL}, there is an exact sequence
$$0 \lra H^1\left(\hX; \Omega^2_{\hX}(\log E)(-E)\right) \lra H^0\left(X;T^1_X\right) \lra \widehat{K}_x' \oplus  H^3(L) \lra 0$$
with $H^1(\hX; \Omega^2_{\hX}(\log E)(-E))\spcheck \cong H^2_E(\hX;\Omega^2_{\hX}(\log E))$, and by \cite[Corollary 1.8 and Theorem 2.1(iv)]{FL}, there is  an exact sequence 
$$0\lra H^2_E\left(\hX;\Omega^2_{\hX}(\log E)\right) \lra \widehat{K}_x' \lra \widehat{A}_x \lra 0.$$
Hence $H^1(\hX; \Omega^2_{\hX}(\log E)(-E)) \cong (\widehat{K}_x')\spcheck/\widehat{A}_x\spcheck$. 

\begin{proposition}\label{smallvbig} Let $\rho\colon \hX  \to X$ be a good resolution of $X$ which is obtained by successively blowing up the curves $C_i$ in some order. 
\begin{enumerate}
\item\label{smallvbig-1} There  are isomorphisms $H^1(\hX; \Omega^2_{\hX}) \cong H^1(X'; \Omega^2_{X'}) \cong  H^1(X'; T_{X'})$, compatible with the natural homomorphisms to $H^1(U;T_U) \cong H^0(X; T^1_X)$.
\item\label{smallvbig-2} Let $\widehat{K}_x'$, $H^3(L)$, and $\widehat{A}_x$ be the groups described above. Then $\widehat{K}_x\cong K_x'$, $H^3(L)\cong S_x$, and $\widehat{A}_x \cong A_x$, so that the following diagram commutes:
$$\begin{CD}
0 @>>> \left(\widehat{K}_x'\right)\spcheck/\widehat{A}_x\spcheck @>>>  H^0\left(X;T^1_X\right) @>>>  \widehat{K}_x' \oplus  H^3(L) @>>> 0\\
@. @VVV @| @VVV @. \\
0 @>>> (K_x')\spcheck/A_x\spcheck @>>>  H^0\left(X;T^1_X\right) @>>>  K_x' \oplus  S_x @>>> 0\rlap{.}
\end{CD}$$
\end{enumerate}
\end{proposition}
\begin{proof} \eqref{smallvbig-1}~   By standard results, $R^0\rho_*\Omega^k_{\hX } = \Omega^k_{X'}$, $R^i\rho_*\scrO_{X'} =0$ for $i> 0$,  and $R^i\rho_*\Omega^k_{\hX } =0$ for $i\geq 2$. As for $R^1\rho_*\Omega^k_{\hX }$, first suppose that $C$ is smooth. Then, by  a standard argument (\textit{cf.} \cite[IV(1.2.1)]{Gros}), $R^1\rho_*\Omega^1_{\hX } \cong \scrO_C$ and $R^1\rho_*\Omega^2_{\hX} \cong \Omega^1_C$. In the general case, a successive application of the Leray spectral sequence for the iterated blowup shows that $R^1\rho_*\Omega^1_{\hX }$ has a filtration whose successive quotients are $\scrO_{C_i}$, $1\leq i\leq r$, and similarly $R^1\rho_*\Omega^2_{\hX }$ has a filtration whose successive quotients are $\Omega^1_{C_i}$, $1\leq i\leq r$. In particular, $H^0(X'; R^1\rho_*\Omega^2_{\hX}) = 0$. Thus, by the Leray spectral sequence, $H^1(\hX; \Omega^2_{\hX}) \cong H^1(X'; R^0\rho_*\Omega^2_{X'}) \cong H^1(X'; \Omega^2_{X'})$, and the remaining statements in \eqref{smallvbig-1} are clear.

\smallskip
\eqref{smallvbig-2}~  By the proof of \eqref{smallvbig-1}, $H^1_C(X';R^1\rho_*\Omega^1_{\hX }) = H^0_C(X';R^1\rho_*\Omega^2_{\hX }) =0$, and  $\dim H^0_C(X'; R^1\rho_*\Omega^1_{\hX })  =\dim  H^1_C(X'; R^1\rho_*\Omega^2_{\hX }) = r$. Similarly, $R^0\rho_*\Cee_{\hX } =\Cee_{X'}$, $R^2\rho_*\Cee_{\hX }$ is a successive extension of the $\Cee_{C_i}$, and otherwise $R^i\rho_*\Cee_{\hX }  =0$. In particular, $H^1_C(X';R^1\rho_*\Omega^2_{\hX }) \cong H^2_C(X';R^2\rho_*\Cee_{\hX })$. 

The Leray spectral sequence in local cohomology gives a spectral sequence with $$E_2^{a,b} = H^a_C\left(X';R^b\rho_*\Omega^2_{\hX }\right) \implies H^{a+b}_E\left(\hX; \Omega^2_{\hX }\right).$$
The only nonzero terms are $H^1_C(X';R^1\rho_*\Omega^2_{\hX })$ and $H^2_C(X';R^0\rho_*\Omega^2_{\hX })=H^2_C(X';\Omega^2_{X'})$. It follows that
$H^1_E(\hX; \Omega^2_{\hX }) =0$  and there is a commutative diagram
$$\begin{CD}
0 @>>> H^2_C\left(X';\Omega^2_{X'}\right) @>>> H^2_E\left(\hX; \Omega^2_{\hX }\right) @>>> H^1_C\left(X';R^1\rho_*\Omega^2_{\hX }\right) @>>> 0 \\
@. @| @| @VV{\cong}V  @.\\
 @. K_x' \oplus \Cee^r @>>> \widehat{K}_x' \oplus H^4_E\left(\hX \right) @>>> \Cee^r\rlap{.} @.
\end{CD}$$
Here, we use the Leray spectral sequence to also conclude that there is an exact sequence
$$\begin{CD}
0 @>>> H^4_C( X') @>>> H^4_E\left(\hX \right) @>>> H^2_C\left(X';R^2\rho_*\Cee_{\hX }\right) @>>> 0 \\
@. @| @. @VV{\cong}V  @.\\
 @. \bigoplus_i\Cee[C_i] @. @. H^1_C\left(X';R^1\rho_*\Omega^2_{\hX }\right)\rlap{.} @.  
 \end{CD}$$
Tracing through the identifications gives $\widehat{K}_x'\cong K_x'$. By definition,  $H^3(L)\cong \bigoplus_i\Cee[C_i] = S_x$. The remaining identification of $\widehat{A}_x$ with $A_x$ is similar, using the Leray spectral sequence to conclude that $H^3_E(\hX;\scrO_{\hX }) \cong H^3_C(X';\scrO_{X'})$ and $H^3_E(\hX;\Omega^1_{\hX }) \cong H^3_C(X';\Omega^1_{X'})$, compatibly with $d$.
\end{proof}

Recall that a $\Cee^*$ action on the germ of an analytic space $(X,x)$ is \textsl{good} if the weights of the induced action on the Zariski tangent space of the fixed point $x$ are all positive. Applying  the case of a good (divisorial) resolution, \textit{cf.} \cite[Theorem 2.1(iv)]{FL}, and using the fact that $X$ is a hypersurface singularity, hence a local complete intersection singularity, we have the following corollary. 

\begin{corollary} The germ $(X,x)$  has a good $\Cee^*$ action if and only if  $A_x= 0$, which holds if and only if $a=0$, which holds if and only if the spectral sequence with $E_1^{p,q} = H^q_C(X';\Omega^p_{X'}) \Rightarrow \mathbb{H}_C^{p+q}(X'; \Omega^\bullet_{X'}) = H_C^{p+q}(X')$ degenerates at $E_2$. 
\end{corollary}

The invariants $b$ and $a$ can also be described in terms of the Du Bois invariants $b^{p,q}(X,x)$ of the singularity $X$, using \cite[Theorem 2.1(iv), (vi)]{FL}. 

\begin{corollary} We have $b = b^{1,1}(X,x)$ and $b-a = b^{2,1}(X,x)$. 
\end{corollary}

\begin{remark}\label{rem-dim} In particular, by Corollary~\ref{smallnumerics}\eqref{sn-3} and the fact that $\ell = \ell^{2,1} =r$, we recover the result of Steenbrink, \textit{cf.} \cite[Theorem 4]{SteenbrinkDB},  that
  $$\dim H^0\left(X;T^1_X\right) = b^{1,1}(X,x) + b^{2,1}(X,x) + \ell^{2,1},$$
 the inequalities  due to Namikawa, \textit{cf.}   \cite[Theorem 1]{NamikawadefCY}, 
$$b^{1,1}(X,x) + \ell^{2,1}\leq \dim H^0\left(X;T^1_X\right)   \leq 2b^{1,1}(X,x) + \ell^{2,1},$$
as well as the statement that $\dim H^0(X;T^1_X)  =2 b^{1,1}(X,x) + \ell^{2,1}$ if and only if $(X,x)$ is weighted homogeneous (since  $(X,x)$ is a local complete intersection).   As shown in the papers of Steenbrink and Namikawa cited above and \cite[Theorem 2.1(iv)]{FL}, these results  hold  more generally  for  isolated rational Gorenstein singularities of dimension $3$. 
\end{remark}

\begin{example}\label{ex-a2n-1} Consider the $A_{2n-1}$ singularity $x^2+y^2+z^2+ w^{2n}$ (a compound $A_1$ singularity). Here, $T^1_X \cong \Cee[w]/(w^{2n-1})$ has dimension $2n-1$ and $r=1$. A calculation shows that $\dim H^0(X;R^1p_*T_{X'}) = n-1$. Hence $\dim H^2_C(X';\Omega^2_{X'}) = n$ by Lemma~\ref{newlemma41}. Also note  that we can deform $\hX$ so that  the exceptional curve $C$ breaks up into a union of $n$ curves with normal bundle $\scrO_{\Pee^1}(-1) \oplus \scrO_{\Pee^1}(-1) $.  This deformation then blows down to a deformation of $X$  to the union of $n$ ordinary double points; \textit{i.e.} $\delta = n$ in the notation of Proposition~\ref{deltaconj}.
\end{example}

\begin{example}\label{ex-cA} Consider the compound $A_{n-1}$ singularity $x^2+y^2 + f(z,w)$, where $f(z,w)=\prod_{i=1}^n(z+\lambda_iw)$ defines a plane curve which is the union of $n$ distinct lines meeting at the origin. An easily computable example is $f(z,w) = z^n-w^n$. Thus
$$T^1_X\cong \Cee[z,w]/\left(z^{n-1}, w^{n-1}\right).$$
A calculation shows that $\dim T^1_X = (n-1)^2$. Note that, for $n\geq 4$,   $(X,0)$ has nontrivial equisingular deformations (there are local moduli). This is also reflected in the fact that there are nontrivial weight zero deformations for $n\geq 4$ (and nontrivial positive weight deformations for $n\geq 5$).  By Corollary~\ref{smallnumerics}, since $a=0$ and $r=n-1$, $\dim H^0(X;R^1p_*T_{X'}) =(n-1)(n-2)/2$ and $\dim H^1_C(X';\Omega^2_{X'}) = n(n-1)/2$. 

The surface $X_0$ defined by $w=0$ is an $A_{n-1}$ singularity, and the inverse image $X'_0$ in $X'$ is a resolution of singularities. Moreover, all of the components $C_i$ in $X'$ have normal bundle $\scrO_{\Pee^1}(-1)\oplus \scrO_{\Pee^1}(-1)$. By general results (\textit{e.g.} \cite{Brieskorn}, \cite{Artin}, or the paper by Pinkham \cite{Pinkham}), there is a morphism of functors $\mathbf{Def}_{X_0'} \to \mathbf{Def}_{X_0}$ with the following property: If  $T$ is the analytic germ prorepresenting the functor $\mathbf{Def}_{X_0}$ and $\widetilde{T}$ is the germ prorepresenting $\mathbf{Def}_{X_0'}$, then the induced morphism $\widetilde{T} \to T$ is a Galois cover of smooth germs with Galois group the Weyl group of the corresponding root system, in this case $A_{n-1}$. The inverse image in $\widetilde{T}$ of the discriminant locus in $T$ consists of $\delta$ hyperplanes, corresponding to keeping one of the $\delta$ positive roots the class of an irreducible effective  curve.

This is in agreement with Proposition~\ref{deltaconj} because in this case $\Dis \delta =\tbinom{n}{2}$ (one can deform the union of $n$ concurrent lines to a union of $n$ lines meeting transversally).  

Specializing to the case $n=5$, and hence $r=4$, we have $\dim H^0(X;T^1_X) = 16$. We can deform the singularity $x^2+y^2 + z^5-w^5$ in the weight $1$ direction to $X_t$  which is defined by
$$x^2+y^2 + z^5-w^5 + tz^3w^3.$$
A calculation shows that $\dim H^0(X_t; T^1_{X_t}) = 15$ if $t\neq 0$. In particular, $a\neq 0$ in this case (in fact $a=1$), and the spectral sequence with $E_1$ page $H^q_C(X'_t;\Omega^p_{X'_t}) \Rightarrow H^{p+q}_C(X'_t)$ does not degenerate at $E_2$ for $t\neq 0$. 
\end{example}

\section{A noncrepant example}\label{Section5}
In this final section, we consider a noncrepant example $\hX$, the blowup of a small resolution with exceptional set a smooth curve $C$ along the curve $C$. Following the discussion of the introduction, there are homomorphisms $H^1(\hX; T_{\hX}) \to H^1(U, T_U) \cong H^0(X;T^1_X)$ and $H^1(\hX; \Omega^2_{\hX}) \to  H^1(U, T_U) \cong H^0(X;T^1_X)$. The image of $H^1(\hX; \Omega^2_{\hX})$ is a birational invariant, \textit{i.e.} is independent of the choice of a good resolution, and is identified with the image of $H^1(X';T_{X'})$ by Proposition~\ref{smallvbig}\eqref{smallvbig-1}. On the other hand,  the image  of $H^1(\hX; T_{\hX})$ also has  geometric meaning (it is the tangent space to the ``simultaneous resolution locus'' for the resolution $\pi\colon \hX \to X$), and the map $H^1(\hX; T_{\hX}) \to  H^0(X;T^1_X)$ factors through the natural map $H^1(\hX; T_{\hX}) \to  H^1(\hX; \Omega^2_{\hX})$.  Our goal in this section is to   explicitly compare the image  of $H^1(\hX; T_{\hX})$ in $H^0(X;T^1_X)$ with that of $H^1(\hX; \Omega^2_{\hX})\cong H^1(X';T_{X'})$. More generally we compare $\mathbf{Def}_{\hX}$ and $\mathbf{Def}_{X'}$. While this example is somewhat special, similar techniques will handle other examples, such as the natural good resolution of the $A_2$ singularity defined by $x^2+y^2 + z^2 + w^3$.

We begin with a general result.   

\begin{lemma}\label{firstLeray} Let $g\colon \hX \to X'$ be the blowup of a smooth threefold $X'$ along a smooth compact curve $C$, with exceptional divisor $E$. Then $R^1g_*T_{\hX}  =R^1g_*T_{\hX}(-E)  = 0$, and there is an exact sequence
 $$0\lra R^0g_*T_{\hX} \lra T_{X'} \lra N_{C/X'} \lra 0.$$
 \end{lemma}
 \begin{proof}  For the first statement, we must show that $(R^1g_*T_{\hX})_t=0$ for all $t\in X'$, and we may as well assume that $t\in C$. By the formal functions theorem, it suffices  to show that
 $$\varprojlim_nH^1\left(nf; T_{\hX}|nf\right) =0,$$
 where $f$ is a fiber of $g\colon \hX \to X'$ over a point $t\in C$, $I_f$ is the ideal sheaf defining the reduced scheme $f$, and $nf$ is the scheme defined by $I_f^n$.  From the exact sequence
 $$0 \lra N_{f/E} \lra N_{f/\hX} \lra N_{E/\hX}|f \lra 0$$
 and the fact that $N_{f/E} \cong \scrO_f$ and $N_{E/\hX}|f \cong \scrO_f(-1)$, we see that 
 $$N_{f/\hX} \cong \scrO_f \oplus \scrO_f(-1).$$
 Because $f$ is a local complete intersection, there is an  exact sequence
 $$0 \lra \Sym^n\left(I_f/I_f^2\right) \lra \scrO_{(n+1)f} \lra \scrO_{nf} \lra 0,$$
 where $I_f/I_f^2$ is the conormal bundle. Hence
 $$\Sym^n\left(I_f/I_f^2\right) = \scrO_f \oplus \scrO_f(1) \oplus \cdots \oplus \scrO_f(n).$$
 Then from the normal bundle sequence
 $$0 \lra T_f \lra T_{\hX}|f \lra N_{f/\hX} \lra 0$$
 and the fact that $T_f \cong \scrO_f(2)$,
 it follows easily that $H^1(nf; T_{\hX}|nf) =0$ for all $n\geq 1$, hence that $R^1g_*T_{\hX} =0$. The proof that 
 $R^1g_*T_{\hX}(-E) =0$ is similar, using $\scrO_{\hX}(-E)|f =\scrO_f(1)$.
 
 To see the statement about $R^0g_*T_{\hX}$, there is an exact sequence
 $$0 \lra T_{\hX} \lra g^*T_{X'} \lra i_*T_{E/C}(E) \lra 0,$$
 where $i\colon E \to \hX$ is the inclusion. Hence there is an exact sequence
\begin{equation*}
0 \lra g_*T_{\hX} \lra g_*g^*T_{X'} \lra g_*i_*T_{E/C}(E) \lra R^1g_*T_{\hX} =0. \tag{$*$}
\end{equation*}
 Then $g_*g^*T_{X'} = T_{X'}$ and $g_*i_*T_{E/C}(E) = r_*T_{E/C}(E)$, where $r\colon E\to C$ is the projection. Also, we have the Euler exact sequence
 $$0 \lra \scrO_E \lra r^*N_{C/X'}(1) \lra  T_{E/C} \lra 0.$$
 Thus, using $\scrO_E(E) = \scrO_E(-1)$, we get
 $$0 \lra \scrO_E(E)  \lra r^*N_{C/X'} \lra  T_{E/C}(E) \lra 0.$$
 Taking $r_*$ and using $R^1r_*\scrO_E(E) =0$  gives $r_*T_{E/C}(E)  = N_{C/X'}$. Thus ($*$) becomes
 $$0 \lra g_*T_{\hX} \lra T_{X'} \lra N_{C/X'} \lra 0,$$
 as claimed.
  \end{proof}
  
\begin{corollary}\label{firstLeraycor} Suppose as above that $g\colon \hX \to X'$ is the blowup of a smooth threefold along a smooth compact curve $C$ and that moreover $p\colon X'\to X$ is a  resolution of an isolated singular point $x$, with $C\subseteq p^{-1}(x)$. Suppose in addition either that $p\colon X'\to X$ is a small resolution of $X$ or that it is a good resolution with $\deg N_{C/E_i}< 0$ for every component $E_i$ of\, $E=p^{-1}(x)$ containing $C$. Then there is an exact sequence 
$$0\lra   H^0\left(C; N_{C/X'}\right) \lra H^1\left(\hX; T_{\hX}\right) \lra H^1\left(X'; T_{X'}\right) \lra H^1\left(C; N_{C/X'}\right).$$
Hence, if  $p\colon X'\to X$ is an equivariant resolution of\, $X$, for example if $p$ is a small resolution, then so is $\pi = p\circ g\colon \hX \to X$. 
\end{corollary}
\begin{proof} By the Leray spectral sequence and Lemma~\ref{firstLeray}, $H^i(\hX; T_{\hX}) = H^i(X';R^0g_*T_{\hX})$ for all $i$. The exact sequence of the statement then follows from the exact sequence for $R^0g_*T_{\hX}$ in Lemma~\ref{firstLeray},  except for the injectivity on the left. If  $H^0(C; N_{C/X'}) \to H^1(\hX; T_{\hX})$ is not injective, then  the map $H^0(X'; T_{X'}) \to H^0(C; N_{C/X'})$ is nonzero. Thus there exists a nonzero element of $H^0(C; N_{C/X'})$ which lifts to $\theta \in H^0(X'; T_{X'})$. Exponentiating the vector field $\theta$, we see that $C$ moves in a one-parameter family in $X'$. However, given the contraction $p\colon X'\to X$, every such family must be contained in the exceptional set of $p$. This is clearly impossible if $p$ is a small resolution or if $C$ does not move in a family in some component $E_i$ of the exceptional set $E$. In particular, if $\deg N_{C/E_i}< 0$ for every component $E_i$ of $E$ containing $C$, then $C$ does not move inside any $E_i$. 

To see the last statement, the above shows that the map $H^0(\hX; T_{\hX}) \to H^0(X'; T_{X'})$ is surjective. Equivalently, $R^0\pi_*T_{\hX} \to R^0p_*T_{X'}$ is surjective, and it is clearly injective, hence an isomorphism. Since by assumption $R^0p_*T_{X'} = T^0_X$, this says that $R^0\pi_*T_{\hX} =T^0_X$, \textit{i.e.} $\pi$ is equivariant.  
\end{proof}

For the rest of this section, $(X,x)$ is a threefold $A_{2n-1}$ singularity, $p\colon X'\to X$ is a small resolution with exceptional curve $C$, and $g\colon \hX \to X'$ is the blowup of the curve $C$, with exceptional divisor~$E$. Then $\pi = p\circ g\colon \hX \to X$ is a noncrepant resolution of $X$, and it is equivariant by Corollary~\ref{firstLeraycor}. If $n =1$, then $H^i(\hX;T_{\hX}) = H^i(X';T_{X'}) = 0$ for $i= 1,2$, and both $\mathbf{Def}_{\hX}$ and $\mathbf{Def}_{X'}$ are (represented by) a single point. Thus, we shall always assume that $n\geq 2$, so that $N_{C/X'} \cong \scrO_C \oplus \scrO_C(-2)$ and $E\cong \mathbb{F}_2$. Note that $K_{\hX} = \scrO_{\hX}(E)$ and $K_E = K_{\hX}\otimes \scrO_{\hX}(E)|E = \scrO_{\hX}(2E)|E$. As $K_E = \scrO_E(-2\sigma -4f)$, where $\sigma$ is the negative section on $E$ and $f$ is the class of a fiber, $\scrO_{\hX}(E)|E = \scrO_E(-\sigma -2f)$. In particular, $N_{E/\hX}\spcheck =\scrO_{\hX}(-E)|E$ is effective, nef, and big, and $H^i(E; N_{E/\hX})=0$ for all $i$ since 
$$H^2\left(E; N_{E/\hX}\right)\spcheck =H^0\left(E; K_E\otimes \scrO_{\hX}(-E)|E\right) = H^0(E;\scrO_E(-\sigma -2f)) =0.$$  

  \begin{corollary}\label{Cor5.3}  Under the above assumptions, the following hold: 
  \begin{enumerate}
  \item\label{Cor5.3-1} We have $H^2(\hX; T_{\hX}) =  H^2(\hX; T_{\hX}(-E)) = 0$. In particular, $\mathbf{Def}_{\hX}$ is unobstructed of dimension $\dim H^1(\hX; T_{\hX})$.
  \item\label{Cor5.3-2} We have $H^1(\hX; T_{\hX}(-\log E)) \cong H^1(\hX; T_{\hX})$,  and the natural map
 $$H^1(\hX; T_{\hX}(-\log E)) \lra H^1(E;T_E)$$
 is surjective. Hence $E$ is a stable submanifold of\, $\hX$, deformations of\, $\hX$ are versal for deformations of $E$, and there exist small deformations of\, $\hX$ for which $E$ deforms to $\mathbb{F}_0$.
 \item\label{Cor5.3-3}  For all $i$, $H^i(\hX; T_{\hX}) = H^i(X';R^0g_*T_{\hX})$, and there is an exact sequence
  $$0\lra \Cee = H^0\left(C; N_{C/X'}\right) \lra H^1\left(\hX; T_{\hX}\right) \lra H^1\left(X'; T_{X'}\right) \lra H^1\left(C; N_{C/X'}\right) \lra 0.$$
  Thus $\dim   H^1(\hX; T_{\hX}) = \dim H^1(X'; T_{X'}) = n-1$.
  \end{enumerate}
  \end{corollary}
  \begin{proof} \eqref{Cor5.3-1}~  To see that $H^2(\hX; T_{\hX}) = 0$, it suffices to show that $R^2\pi_*T_{\hX} =0$. In the Leray spectral sequence with $E_2^{a,b} = R^ap_*R^bg_*T_{\hX}\Rightarrow R^{a+b}\pi_*T_{\hX}$, all possible terms contributing to $R^2\pi_*T_{\hX}$ are $0$, either for dimension reasons or because $R^1p_*R^1g_*T_{\hX} =0$ by Lemma~\ref{firstLeray}. Thus $R^2\pi_*T_{\hX} =0$. The proof for  $H^2(\hX; T_{\hX}(-E))$ is similar.

\smallskip
\eqref{Cor5.3-2}~ From the exact sequence
  $$0 \lra T_{\hX}(-\log E) \lra T_{\hX} \lra N_{E/\hX} \lra 0$$
   and the fact that $H^i(E; N_{E/\hX})=0$ for all $i$ (apply Leray to the morphism $r\colon E \to C$), we have  $H^i(\hX; T_{\hX}(-\log E)) \cong H^i(\hX; T_{\hX})$ for all $i$. Thus in particular $H^2(\hX; T_{\hX}(-\log E)) = 0$ and $H^1(\hX; T_{\hX}(-\log E)) \cong H^1(\hX; T_{\hX})$. Finally, from the exact sequence
   $$0 \lra T_{\hX}(-E) \lra T_{\hX}(-\log E) \lra T_E \lra 0$$
   and the vanishing of $H^2(\hX; T_{\hX}(-E))$, we see that $H^1(\hX; T_{\hX}(-\log E)) \to H^1(E;T_E)$
 is surjective.

\smallskip
\eqref{Cor5.3-3}~ This follows from Corollary~\ref{firstLeraycor} and the fact that $H^2(\hX; T_{\hX}) =0$.
  \end{proof}
  
  \begin{remark}\leavevmode
    \begin{enumerate}[wide]
\item  By \eqref{Cor5.3-3} above, the images of $H^1(\hX; T_{\hX})$ and $H^1(\hX; \Omega^2_{\hX})$ in $H^0(X;T^1_X)$ are different since by Proposition~\ref{smallvbig}, the image of $H^1(\hX; \Omega^2_{\hX})$ is that of $H^1(X'; \Omega^2_{X'}) = H^1(X'; T_{X'})$, and this image is strictly larger than that of $H^1(\hX; T_{\hX})$. 
 
\item  The functors  $\mathbf{Def}_{\hX}$ and $\mathbf{Def}_{X'}$ are both smooth of dimension $n-1$, but the differential of the corresponding morphism of functors $\mathbf{Def}_{\hX} \to \mathbf{Def}_{X'}$, \textit{i.e.} the  induced map on Zariski tangent spaces, is not an isomorphism at $0$: It has a $1$-dimensional kernel and cokernel. We will  describe the morphism  $\mathbf{Def}_{\hX} \to \mathbf{Def}_{X'}$ explicitly.
  \end{enumerate}
 \end{remark} 
  
  First, let $\mathbf{Def}_{\hX, E}$ denote the functor of deformations of the pair $(X,E)$: For a germ $(S,s_0)$, an element of $\mathbf{Def}_{\hX, E}(S,s_0)$ consists of a deformation $\widehat{\mathcal{X}}$ of $\hX$ over $S$, together with an effective Cartier divisor $\mathcal{E}$ of $\widehat{\mathcal{X}}$, flat over $S$ and restricting to $E$ over $s_0$. The functor $\mathbf{Def}_{X',C}$ is defined similarly; the objects over $S$ are pairs $(\mathcal{X}', \mathcal{C})$, where $\mathcal{C}$ is flat over $S$ and restricts to the reduced subscheme $C$ of $X'$. In particular, as we are only considering germs of spaces, $\mathcal{C}$ is smooth over $S$ with all fibers irreducible.
  
  \begin{proposition} We have  $\mathbf{Def}_{\hX} \cong \mathbf{Def}_{\hX, E} \cong \mathbf{Def}_{X',C}$, and the morphism $\mathbf{Def}_{\hX} \to \mathbf{Def}_{X'}$ is the same under the above identification as the forgetful morphism $\mathbf{Def}_{X',C} \to \mathbf{Def}_{X'}$.
  \end{proposition} 
  \begin{proof} It is a standard result that the tangent space to $\mathbf{Def}_{\hX, E}$ is $H^1(\hX; T_{\hX}(-\log E))$  (\textit{cf.} Section~1) and the obstruction space is $H^2(\hX; T_{\hX}(-\log E))=0$. Thus, $\mathbf{Def}_{\hX, E}$ is smooth by Corollary~\ref{Cor5.3}(ii), and the first statement of the proposition is the isomorphism $H^1(\hX; T_{\hX}(-\log E)) \cong H^1(\hX; T_{\hX})$. For the second, given a pair $(\widehat{\mathcal{X}}, \mathcal{E})$ in  $\mathbf{Def}_{\hX, E}(S,s_0)$, it is easy to check that the morphism $r\colon E \to C$ extends to a morphism $\mathcal{E} \to \mathcal{C} \cong C\times S$ (note that $C\cong \Pee^1$ is rigid) and that $\mathcal{E}$ can be blown down to a subspace $\mathcal{C} \cong C\times S \subseteq \mathcal{X}'$. Conversely, given a pair $(\mathcal{X}', \mathcal{C})$ over $S$, let $\widehat{\mathcal{X}}$ be the blowup of $\mathcal{X}'$ along  $\mathcal{C}$, and let $ \mathcal{E}$ be the exceptional divisor. This gives  two morphisms of functors $\mathbf{Def}_{\hX, E} \to \mathbf{Def}_{X',C}$ and $\mathbf{Def}_{X',C}  \to \mathbf{Def}_{\hX, E}$ which are clearly inverses. Hence $\mathbf{Def}_{\hX, E} \cong \mathbf{Def}_{X',C}$. 
  \end{proof} 
  
  To put the above in more manageable form, we give an explicit description of $\mathbf{Def}_{X',C}$. First, we recall the basics about deformations of $X$ and $X'$. Let $Z$ be the germ of the standard ordinary double point in dimension $2$, given by $x^2+y^2+z^2=0$, and let $Z'$ be the resolution of singularities of $Z$, with $C\subseteq Z'$ the exceptional set. Then $\mathbf{Def}_Z$ is represented by the germ $(\Cee, 0)$, with coordinate $t$ and universal family $\mathcal{Z}\to (\Cee, 0)$ given by $x^2+ y^2+z^2+t =0$. Likewise, $\mathbf{Def}_{Z'}$ is represented by the germ $(\Cee, 0)$, with coordinate $u$  and universal family $\mathcal{Z}'\to (\Cee, 0)$ given by a choice for a small resolution of $x^2+ y^2+z^2+u^2 =0$. The threefold $X'$ is isomorphic in a neighborhood of $C$ to $f^*\mathcal{Z}'$, where $f\colon (\Cee, 0) \to (\Cee, 0)$ is given by $u = f(w) = w^n$. A deformation $\mathcal{X}'$ of $X'$ over a germ $(S,s_0)$ corresponds to a  morphism $F\colon (\Cee, 0) \times S \to (\Cee, 0)$ with $F(w, s_0) = f(w) = w^n$, with $\mathcal{X}' = F^*\mathcal{Z}'$. A polynomial $F$ restricting to $w^n$ is analytically equivalent to one of the form $w^n + b_{n-2}(s)w^{n-1} + \cdots + b_0(s)$.  In particular, $\mathbf{Def}_{X'}$ is represented by the germ of the affine space $(\Cee^{n-1},0)$, with coordinates $(b_{n-2}, \dots , b_0)$.  Also note that $F^{-1}(0) \to S$ is a finite cover of degree $n$.
  
  \begin{lemma}  There is an isomorphism of functors from $\mathbf{Def}_{X',C}$ to $\mathbf{F}$, where, for $(S,s_0)$ the germ of an analytic space, $\mathbf{F}(S)$ is the set of  pairs $(F, \sigma)$, where as above $F\colon (\Cee, 0) \times S \to (\Cee, 0)$ is a morphism such that  $F(w, s_0) = f(w) = w^n$ and $\sigma \subseteq F^{-1}(0)$ is a section of the finite cover $F^{-1}(0) \to S$, or equivalently a morphism $\lambda \colon S \to (\Cee,0)$ such that $F(\lambda(s), s)$ is identically  $0$. Moreover, via this isomorphism, the morphism $\mathbf{Def}_{X',C} \to \mathbf{Def}_{X'}$ corresponds the forgetful map $(F,\sigma) \in \mathbf{F}(S) \mapsto F$.   
  \end{lemma} 
   \begin{proof} Given an object $(F, \sigma)$ of $\mathbf{F}(S)$, the morphism $F$ defines $\mathcal{X}'$ in $\mathbf{Def}_{X'}(S)$ corresponding to the morphism $F\colon (\Cee, 0) \times S \to (\Cee, 0)$ as above, and a Cartesian diagram
  $$\begin{CD}
  \mathcal{X}' @>>> \mathcal{Z}'\\
  @VVV @VVV\\
  (\Cee, 0) \times S @>{F}>> (\Cee, 0).
  \end{CD}$$ 
  Note that $\mathcal{X}'|F^{-1}(0) = Z'\times F^{-1}(0) \subseteq \mathcal{X}'$, and thus 
  $$ C\times F^{-1}(0) \subseteq \mathcal{X}'|F^{-1}(0)\subseteq \mathcal{X}',$$
  compatibly with the projection to $S$. Given the section $\sigma \subseteq F^{-1}(0)$, define 
  $$\mathcal{C} = C\times \sigma \subseteq C\times F^{-1}(0)  \subseteq 
  \mathcal{X}'.$$
  Thus the pair $(F,\sigma)$ defines a deformation of $(X', C)$ over $S$, and hence an element of $\mathbf{Def}_{X',C}$.   
  
 Conversely,  suppose that we are given a pair $(\mathcal{X}', \mathcal{C})\in \mathbf{Def}_{X',C}(S)$, and let $F\colon (\Cee, 0) \times S \to (\Cee, 0)$ be the  morphism  corresponding  to $\mathcal{X}'$ in $\mathbf{Def}_{X'}(S)$.  Note that $C$ does not deform in $\mathcal{Z}'$, even to first order. Thus  $\mathcal{C}\cong C \times S \subseteq C\times F^{-1}(0)$, compatibly with the projection to $S$, so that the projection of $\mathcal{C}$ onto the second factor $F^{-1}(0)$ defines a section $\sigma$ of the morphism $F^{-1}(0) \to S$. Clearly, the two constructions  $\mathbf{F}(S) \to \mathbf{Def}_{X',C}(S)$ and $\mathbf{Def}_{X',C}(S) \to \mathbf{F}(S)$ are mutual inverses and are functorial under pullback. This defines the isomorphism of functors, and the final statement is clear from the construction.
   \end{proof}
   
   Explicitly, with $P(w;\mathbf{b})= w^n + \sum_{i=0}^{n-2}b_iw^i$, where $\mathbf{b}= (b_{n-2}, \dots, b_0)$,  the universal deformation $\mathcal{U}'$ of $X'$ is given as a small resolution of $x^2+y^2 + z^2 + (P(w;\mathbf{b}))^2 =0$.  Consider 
   $$\{(\lambda, b_{n-2}, \dots, b_0) : P(\lambda;\mathbf{b})= 0\}\subseteq (\Cee\times \Cee^{n-1},0).$$
   Note that if $w-\lambda$ is a factor of $P(w;\mathbf{b})$, then
 $$P(w;\mathbf{b}) = (w-\lambda)(w^{n-1} + \lambda w^{n-2} + t_{n-3} w^{n-3} + \cdots + t_1w + t_0) = (w-\lambda)Q(w;\lambda, \mathbf{t}),$$
 say, where $Q(w;\lambda, \mathbf{t}) = w^{n-1} + \lambda w^{n-2} + t_{n-3} w^{n-3} + \cdots + t_1w + t_0$.
Solving explicitly for the coefficients $b_i$, we see
\begin{align*}
b_{n-2} &= -\lambda^2  + t_{n-3},\\
b_i &= -\lambda t_i + t_{i-1}, \quad 1\leq i \leq n-3,\\
b_0 &= -\lambda t_0.
\end{align*}
If $\lambda, t_{n-3}, \dots , t_0$ are coordinates on $\Cee^{n-1}$, this defines a morphism $\Phi\colon \Cee^{n-1} \to \Cee^{n-1}$ by
$$\Phi(\lambda, t_{n-3}, \dots , t_0) = (-\lambda^2  + t_{n-3}, -\lambda t_{n-3} + t_{n-4}, \dots,  -\lambda t_0)= (b_{n-2}, \dots, b_0).$$ 
Solving for $t_i$ in terms of $\lambda$ and the $b_i$ gives 
\begin{align*}
t_i &= \lambda^{n-i-1} + b_{n-2}\lambda^{n-i-3} + \cdots + b_{i+1}, \quad 1\le i \le n-3,\\
-b_0 & = \lambda^n + b_{n-2}\lambda^{n-2} + \cdots + \lambda b_1,
\end{align*}
recovering the fact that $\lambda$ is a root of $P(w;\mathbf{b})$ (and there are exactly $n$ such roots). 
Thus the morphism $\pi_1 \times \Phi\colon \Cee^{n-1} \to \Cee \times \Cee^{n-1}$ defined by 
$$(\lambda, t_{n-3}, \dots , t_0) \longmapsto (\lambda, \Phi(\lambda, t_{n-3}, \dots , t_0))$$
is an isomorphism from $\Cee^{n-1}$ to $\{(\lambda, b_{n-2}, \dots, b_0) : P(\lambda;\mathbf{b})= 0\}$. The germ $(\Cee^{n-1},0)$, with coordinates $\lambda, t_{n-3}, \dots , t_0$, together with the family which is a small resolution of $\Phi^*\mathcal{U}'$, represents $\mathbf{Def}_{X',C}$, and hence after blowing up represents $\mathbf{Def}_{\hX}$. Moreover, $\Phi$ corresponds to the forgetful morphism $\mathbf{Def}_{X',C} \to \mathbf{Def}_{X'}$. 

By the above,  $\Phi$ is finite of degree $n$ and surjective, and is ramified exactly where   $Q(\lambda;\lambda, \mathbf{t})=0$, \textit{i.e.} where $P(w;\mathbf{b})$ has $\lambda$ as a double root (or where the discriminant of $P(w;\mathbf{b})$ vanishes). In fact,  
$$\det \begin{pmatrix} \frac{\partial b_{n-2}}{\partial \lambda} &\cdots & \frac{\partial b_0}{\partial \lambda}\\[1ex]
 \frac{\partial b_{n-2}}{\partial t_{n-3}} &\cdots & \frac{\partial b_0}{\partial  t_{n-3}}\\
 \vdots&\vdots &\vdots\\
 \frac{\partial b_{n-2}}{\partial t_0} &\cdots & \frac{\partial b_0}{\partial  t_0}
 \end{pmatrix} = \pm   Q(\lambda;\lambda, \mathbf{t}).$$
 
 Summarizing the above discussion, then, we have the following.
 
 \begin{theorem}\label{noncrepthm}  Let $(S_{\hX},0)$ and $(S_{X'},0)$ be the germs prorepresenting the functors $\mathbf{Def}_{\hX}$ and $ \mathbf{Def}_{X'}$ respectively. Then the induced morphism $S_{\hX}\to S_{X'}$ is finite of degree $n$, and its differential at the origin has a $1$-dimensional kernel and cokernel. 
 \end{theorem}


\providecommand{\bysame}{\leavevmode\hbox to3em{\hrulefill}\thinspace}

\end{document}